\documentclass{article}
\usepackage{amsmath,amssymb,amsfonts,amsthm,latexsym,amstext,amscd}
\usepackage{epsfig,graphics,graphicx,color}
\usepackage{enumerate}
\usepackage{hyperref}

\newcommand{\bsu}{\boldsymbol{u}}%
\newcommand{\bsl}{\boldsymbol{l}}%
\newcommand{\NN}{\mathbb N}
\newcommand{\ZZ}{\mathbb Z}
\newcommand{\RR}{\mathbb R}
\newcommand{\QQ}{\mathbb Q}

\newcommand{\diag}{{\rm diag}}

\newcommand{\lc}{{\rm lc}}
\newcommand{\FF}{\mathbb{F}}
\newcommand{\KK}{\mathbb{K}}
\newcommand{\kk}{\boldsymbol{k}}
\newcommand{\Fq}{\mathbb{F}_q}
\newcommand{\Fp}{\mathbb{F}_p}

\newtheorem{theorem}{Theorem}
\newtheorem{lemma}{Lemma}
\newtheorem{remark}{Remark}
\newtheorem{example}{Example}
\newtheorem{prop}{Proposition}
\newtheorem{coro}{Corollary}
\newtheorem{defi}{Definition}

\newtheorem{algo}{Algorithm}

\begin{document}

\title{Finding both, the continued fraction and the Laurent series expansion of golden ratio analogs in the field of formal power series}

\author{Roswitha Hofer\thanks{Institute of Financial Mathematics and Applied Number Theory, Johannes Kepler University Linz, Altenbergerstr. 69, 4040 Linz, Austria. e-mail: roswitha.hofer@jku.at}}

\date{\today}
\maketitle

\begin{abstract}
The focus of this paper is on formal power series analogs of the golden ratio. We are interested in both their continued fractions expansions as well as their Laurent series expansions.
Knowing both of them is for instance important for the study of Kronecker type sequences in the theory of uniform distribution. 

Our approach studies the Hankel matrices that are determined using the coefficients of the Laurent series expansions. 

We discover that both matrices in the $\mathcal{L}\mathcal{U}$ decomposition of these Hankel matrices can be described by a simple recursive algorithm based on the continued fraction coefficients of the golden ratio analogs. 

The upper triangular factor $\mathcal{U}$ possesses several nice properties. 
First, its entries in the special case where the golden ratio analog is $[0;\overline{X}]$ can be given by using the Catalan's triangular numbers. Nicely, this relation together with our findings on the Hankel matrices are used to derive combinatorial identities involving Catalan's triangular numbers.  As a side product we obtain a description of the distribution of the Catalan numbers modulo a prime number. 

Second, the upper triangular factor $\mathcal{U}$ is the columnwise composition of the Zeckendorf-type representations of the powers of $X$. This Zeckendorf representation for polynomials is introduced in the style of the Zeckendorf representation of the natural numbers and it is based on the Fibonacci polynomials instead Fibonacci numbers. 

Third, in case of positive characteristic, for each purely periodic golden ratio analog we detect a self-similar pattern in the upper triangular factor $\mathcal{U}$ of its Hankel matrix. We derive this pattern from a binomial type formula for two non-commutative matrices. 

Finally, we exemplary show, how this self similar pattern in the upper triangular factor $\mathcal{U}$ of its Hankel matrix can be used to describe the Laurent series expansion of a golden ratio analog for which the continued fraction coefficients are given. 
\end{abstract}

\noindent{\textbf{Keywords:} 
golden ratio, formal power series, Fibonacci polynomials, Catalan's triangular numbers, LU-decomposition of regular Hankel matrices}\\
\noindent{\textbf{MSC2010:} primary 11B39; secondary 11T99, 11C99.}

\section{Introduction and Motivation}

We remember the golden ratio $\varphi:=\frac{\sqrt{5}+1}{2}$ as a root of the polynomial $p[Y]\in\QQ[Y]$
$$p[Y]=Y^2-Y-1$$
and its fractional part $\{\varphi\}=\frac{\sqrt{5}-1}{2}=1/\varphi$ as a root of the polynomial 
$\overline{p}[Y]\in\QQ[Y]$
$$\overline{p}[Y]=Y^2+Y-1.$$
It is well-known that both $\varphi$ and $1/\varphi$ are quadratic irrationals in $\RR$ which is the completion field of $\QQ$ with respect to the absolute value $|\,\cdot\,|$. 

The golden ratio $\varphi$ and its inverse $1/\varphi$ are in a certain sense optimal real numbers, namely their simple continued fraction coefficients have lowest possible absolute values $1$, as $\varphi=[1;\overline{1}]$ and $1/\varphi=[0;\overline{1}]$, and the latter is contained in the interval $[0,1)$. \\ 
Note that the $n$th convergent to $\varphi$ is $\frac{F_{n+1}}{F_n}$ and to $1/\varphi$ is $\frac{F_{n-1}}{F_n}$, where $F_n$ denotes the $n$th Fibonacci number. The sequence of Fibonacci numbers $(F_n)_{n\geq0}$ is recursively defined via $F_{n}=F_{n-1}+F_{n-2}$ where we choose the starting values $F_0=1$ and $F_{-1}=0$. \\

Now we use -- in analogy to the number field $\QQ$ -- the rational function field $\KK=\kk(X)$ where $\kk$ is an arbitrary field, as for example $\QQ$ or a finite field $\Fq$ with $q$ elements. And the pendant to the absolute value is the non-archimedean absolute value $|\,\cdot\,|_{\infty}$, that is defined using the degree valuation $\nu_{\infty}$ as follows. \\

For every $f\in\kk(X)$ we can find $p(X),\,q(X)\in\kk[X]$ satisfying $q(X)\neq 0$ such that $f=\frac{p(X)}{q(X)}$. We write $q(X)=\sum_{i=0}^ra_iX^i$ with $a_i\in\kk$ and $a_r\neq 0$. Then the degree $\deg(q(X))$ of $q(X)$ is well-known to be $r$. We define $\deg(0):=-\infty$.  \\
The degree valuation $\nu_{\infty}$ of the rational function $f$ then is defined as
$$\nu_{\infty}(f):=\deg(p(X))-\deg(q(X))$$
and the corresponding non-archimedean absolute value $$|f|_{\infty}:=e^{\nu_{\infty}(f)},$$
where $e$ is here the Euler number.\\ 

Now the completion field of $\kk(X)$ with respect to the absolute value $|\,\cdot\,|_{\infty}$ is the field of formal Laurent series (also called formal power series), denoted by $\kk((X^{-1}))$. Each element $L\neq 0$ in $\kk((X^{-1}))$ can be written as 
$$L=\sum_{i= u}^\infty c_iX^{-i}$$
here $c_i \in\kk$ for all $i\in\ZZ$ and $u\in\ZZ$ such that $c_u\neq 0$ and $c_i=0$ for all $i<u$.

The degree evaluation $\nu_{\infty}(L)$ of $L$ is $-u$ and the absolute value $|L|_{\infty}$ of $L$ is $e^{-u}$. \\

In the following we list analogies between $\RR$ and $\kk((X^{-1}))$ that are of use in this paper.  
\begin{enumerate}
	\item Let $b\in\ZZ$ satisfying $b>1$. Then every real number $\alpha$ has a unique $b$-adic expansion of the form
	$$[\alpha]+\sum_{i=1}^\infty c_ib^{-i}$$
	with $c_i\in\{0,1,\ldots,b-1\}$ and infinitely many $c_i\neq b-1$, and $[\alpha]$ is the integer part of $\alpha$ satisfying $\alpha-[\alpha]\in[0,1)$. The latter magnitude is abbreviated to $\{ \alpha\}$ and called the fractional part of $\alpha$. \\
	
	The element $X\in\kk[X]$ satisfies $|X|_\infty=e^1>1$. Every nonzero element $L$ in $\kk((X^{-1}))$, can be uniquely written as 
	$$ [L]+\sum_{i=1}^\infty c_iX^{-i} $$
	with coefficients $c_i\in\kk$, and $[L]=\sum_{i=0}^{-u}{c_{-i}}X^i$ is the polynomial part of $L$ contained in $\kk[X]$. Note that the absolute value of $|L-[L]|_{\infty}$ is smaller than $1$. The element $(L-[L])$ is usually abbreviated to $\{L\}$ and called the fractional part of $L$. 
	
	The set of all possible fractional parts, that is $$\{L\in\kk((X^{-1})):|L|_{\infty}<1\},$$ 
	will be abbreviated to $\overline{\kk((X^{-1}))}$ in the following. 
	
	Note that this set can be alternatively given as 
	$$\left\{\sum_{i=1}^\infty c_iX^{-i}:c_i\in\kk\right\}.$$
	
	The degree evaluation $\nu_\infty$ of $\sum_{i=1}^\infty c_iX^{-i}\neq 0$ satisfies $$\nu_{\infty}\left(\sum_{i=1}^\infty c_iX^{-i}\right)=-\min\left(\{i\in\NN:c_i\neq 0\}\right).$$
	
\item Every real number $\alpha \in [0,1)$ can be expressed in a continued fraction $[0;a_1,a_2,a_3,\ldots]$ satisfying $a_i\in\NN$ for every $i\in\NN$. \\

In analogy, every series $L\in\overline{\kk((X^{-1}))}$ can be expressed in a continued fraction 
$$[0;A_1(X),A_2(X),A_3(X),\ldots]=\cfrac{1}{A_1(X)+\cfrac{1}{A_2(X)+\cfrac{1}{A_3(X)+\ddots} }}$$
with $A_i(X)\in\kk[X]$ satisfying $\deg(A_i(X))\geq 1$ for every $i\in\NN$.
 

\item The $n$th convergent $\alpha_n$ to a non-rational $\alpha\in[0,1)$ is the finite continued fraction $[0;a_1,a_2,\ldots,a_n]$ which can be written as $p_n/q_n$ with $p_n,q_n \in\NN_0$. Both recursively computed via $q_n=a_nq_{n-1}+q_{n-2}$ and $p_n=a_np_{n-1}+p_{n-2}$ for every $n\in\NN$ with starting values $q_{-1}=0,\,p_{-1}=1,\,q_0=1$, and $p_0=0$.  \\
Analogously, the $n$th convergent $L_n\in\kk(X)$ to a non-rational $L\in\overline{\kk((X^{-1}))}$ is the finite continued fraction $[0;A_1(X),A_2(X),\ldots,A_n(X)]$ which can be written as $P_n(X)/Q_n(X)$ with $P_n(X),Q_n(X) \in\kk[X]$. Both recursively computed via $Q_n(X)=A_n(X)Q_{n-1}(X)+Q_{n-2}(X)$  and $P_n(X)=A_n(X)P_{n-1}(X)+P_{n-2}(X)$ for every $n\in\NN$ with initial values $Q_{-1}(X)=0,\,P_{-1}(X)=1,\,Q_0(X)=1$, and $P_0(X)=0$.  

\item Now an interesting fact is that the denominators $q_n$ or $Q_n(X)$ are best approximating in the sense that if 
$$|\alpha-p/q|<1/|q|^2$$ then $p/q$ is a convergent to $\alpha$ and if 
$$\nu_{\infty}(L-P(X)/Q(X))<2\nu_{\infty}\left( \frac{1}{Q(X)}\right)=-2 \deg(Q(X))$$
then $P(X)/Q(X)$ is a convergent to $L$. 

Moreover, we note that $\deg(Q_n)=\sum_{i=1}^n\deg(A_i)=:d_n$, $\nu_{\infty}(L-P_n(X)/Q_n(X))=-d_n-d_{n+1}$ and
$$\nu_{\infty}(L-v(X)/w(X))\geq \nu_{\infty}(L-P_n(X)/Q_n(X))$$
for all $v(X),w(X)\in\kk[X]$ satisfying $0\leq \deg(w(X))< d_{n+1} $.


\item Finally, we point out two further analogies between $\RR$ and the set of formal power series $\kk((X^{-1}))$ in the case where $\kk$ is a finite field. 
 We have the following two classifications of $\alpha\in\RR$:
\begin{itemize}
	\item[-] $\alpha$ is a rational number if and only if for every integer $b> 1$, the $b$-adic expansion of $\alpha$ is eventually periodic or finite. 
\item[-] $\alpha$ is a quadratic irrational number if and only if the continued fraction expansion of $\alpha$ is infinite and ultimately periodic. 
\end{itemize}
In the case where the field $\kk$ is finite, we have for $L\in\kk((X^{-1}))$ that the condition $L\in\kk(X)$ is equivalent to that the series expansion $\sum_{i=u}^\infty c_iX^{-i}$ is eventually periodic or finite. (See \cite[Theorem~1.1]{lasj2000})

Moreover, in case of a finite field $\kk$ we have for $L\in\kk((X^{-1}))\setminus \kk(X)$,  that the continued fraction expansion of $L$ is ultimately periodic if and only if it is quadratic over $\kk(X)$ (cf. e.g. \cite[Theorem~3.1]{lasj2000}). 

For results and discussions in the case where $\kk$ is not finite, confer e.g. \cite{schmidt2000}. 



\end{enumerate}


For more information on continued fractions in power series fields we refer to \cite{lasj2000} and \cite{schmidt2000}.\\

In this paper we introduce analogs of the reciprocal of the golden ratio $1/\varphi$ which possesses the ``optimal'' continued fraction expansion 
$$[0;\overline{1}]=\cfrac{1}{1+\cfrac{1}{1+\cfrac{1}{1+\cfrac{1}{1+\ddots}}}} ,$$
since all coefficient have lowest possible value $1$.  
As there are more choices of partial quotients having lowest possible degree $1$, we define a set of golden ratio analogs in $\overline{\kk((X^{-1}))}$.

\begin{defi}[Set of golden ratio analogs]\label{def:set_GR}
We define the set $\Phi$ of golden ratios in $\overline{\kk((X^{-1}))}$ as $$\Phi:=\left\{[0;u_1X+v_1,u_2X+v_2,u_3X+v_3,\ldots]\,\,|\,\, u_i\in\kk^*, v_i\in\kk \mbox{ for all }i\in\NN\right\}.$$ 
\end{defi}

We are particularly interested in the formal power series representation of specific examples of golden ratio analogs, as e.g. $[0;\overline{X}]$ over special fields $\kk$. 
We would like to point out that for real numbers in $[0,1)$ it is a non-trivial problem to determine both, the explicit $b$-adic representation with a fixed base $b$, say $2$ or $10$, and its continued fraction coefficients. So far there are only a few examples, where we have information on both representations (see e.g. Shallit \cite{shallit1,shallit2}). We mention here the number investigated in \cite{shallit1}.

\begin{example}[Shallit]\label{ex:shallit}{\rm
Let $b\geq 3$ be an integer and $\sigma:=\sum_{k=0}^\infty b^{-2^k}$. Then Shallit \cite{shallit1} discovered a certain pattern in the continued fraction of $\sigma$ as the limit of the following recursion, e.g., in the case where $b=3$ the pattern can be described as follows:
$B(1)=[0;2,4],\,B(2)=[2,5,3,4]$, and for $k\geq 2$ with $B(k)=[0;a_1,a_2,\ldots,a_{n-2},a_{n-1},a_n]$ let $$B(k+1)=[0;a_1,a_2,\ldots,a_{n-2},a_{n-1},a_n+1,a_n-1,a_{n-1},a_{n-2},\ldots,a_2,a_1].$$
Now $\sigma$ in base $3$ is $\lim_{k\to \infty} B(k)$. \\
Note for $b=4$ we may write $2\sigma$ as $\sum_{k=1}^\infty2^{-2^k+1}$. }
\end{example}

The question for the formal series expansion of golden ratio analogs is for instance relevant for the investigation of so-called Kronecker type sequences over $\FF_q((X^{-1}))$ where $\Fq$ is the finite field with $q$ elements. Kronecker type sequences, that are introduced as an analogon to the ordinary Kronecker sequence $(\{n\alpha\})_{n\geq 0}$ with $\alpha\in\RR$, are actively investigated in the theory of uniform distribution (cf e.g.\cite{larnie1993}). For the sake of completeness we give the definition of Kronecker type sequences.



\begin{defi}[Kronecker type sequence]\label{def:KTseq}
Let $\FF_q$ be a finite field and $L\in \overline{\Fq((X^{-1}))}$. For $n\in\NN_0$, the $n$th element $x_n$ of the Kronecker type sequence determined by $L$ is obtained by the following algorithm:
\begin{enumerate}
	\item Expand $n$ in base $q$, i.e. $n=\sum_{i=0}^\infty n_iq^i$ with $n_i\in\{0,1,\ldots,q-1\}$, then associate $n$ with the polynomial $n(X)\in\FF_q[X]$ by setting $$n(X)=\sum_{i=0}^\infty \eta(n_i)X^i,$$
	where $\eta:\{0,1,\ldots,q-1\}\to \FF_q$ is a fixed bijection that maps $0$ to $0$. 
	\item Compute the product $n(X) L$ and take the series expansion of its fractional part,
	$$\{n(X)L\}=\sum_{i=1}^\infty u_iX^{-i}.$$
	\item Finally, set $$x_n=\sum_{i=1}^\infty \eta^{-1}(u_i)q^{-i},$$
	which is always an element of the interval $[0,1]$. 
	For short we write this sequence as $([\{n(X)L\}]_q)_{n\geq 0}$. 
\end{enumerate}
\end{defi}
\begin{remark}\label{rem:KTseq}{\rm
Note that we can also use the concept of an $\NN\times \NN$ generating matrix $\mathcal{M}_L$ over $\FF_q$ and the digital method, that was introduced in \cite{nie1987}, to define the Kronecker type sequence determined by $L$. The digital method works as follows. 
Instead of a polynomial $n(X)$ associate a vector  $\vec{n}=(\eta(n_0),\eta(n_1),\ldots)^T\in\FF_q^{\NN}$. Define the generating matrix 
$$\mathcal{M}_L:=(m_{i,j})_{i\geq 1,j\geq 1}\in\FF_q^{\NN\times \NN}$$
with entries $m_{i,j}=:c_{i+j-1}$, where the $c_{i}$ are the coefficients in the series expansion of $L=\sum_{i=1}^\infty c_i X^{-i}$. Then, compute $\mathcal{M}_L\cdot \vec{n}=:\vec{y}_n=(y_{n,1},y_{n,2},\ldots)^T\in\FF_q^{\NN}$. Finally, set $$x_n=\sum_{i=1}^\infty \eta^{-1}(y_{n,i})q^{-i}.$$
Note that the generating matrix of any Kronecker type sequence is a Hankel matrix, since $m_{i,j}=m_{i+k,j-k}$ for all $i,j,k\in\NN$ satisfying $k< j$. 
While the generating matrix $\mathcal{M}_L$ can be defined over an arbitrary field $\kk$, the construction algorithm of the sequence in $[0,1)$ relies on the finiteness of the field.  
}
\end{remark}

In the theory of uniform distribution the so-called $L_2$-discrepancy of Kronecker type sequences is of particular interest as the $L_2$-discrepancy of ordinary Kronecker sequences is well investigated (cf. for example \cite{Bilyk}). The main method for investigating the distribution of Kronecker type sequences applies Walsh functions whereas for the ordinary Kronecker sequence $(\{\alpha n\})_{n\geq 0}$ trigonometric functions are used. For the Walsh function approach detailed information on the formal series coefficients is helpful. It is known that one-dimensional Kronecker type sequences $([\{n(X)L\}]_q)_{n\geq 0}$ are in a certain sense optimal if all continued fraction coefficients of $L$ have degree $1$ (see e.g. \cite[Theorem~4.48]{niesiam}). Therefore the series expansions of golden ratio analogs in $\overline{\kk((X^{-1}))}$, where $\kk$ is a finite field, deserve particular attention. \\
As already noted in the abstract our approach studies the Hankel matrices that are the generating matrices of the Kronecker type sequences and which are defined using the coefficients of the Laurent series expansions.

\begin{defi}[Hankel matrix associated with $L$]
Let $L\in\overline{\kk((X^{-1}))}$ with Laurent series expansion $L=\sum_{i=1}^\infty c_iX^{-i}$. 
We define the Hankel matrix $\mathcal{M}_{L}=(m_{i,j})\in\kk^{\NN\times\NN}$ related to $L$ by setting $m_{i,j}=:c_{i+j-1}$ for all $i,j\in\NN$. 
\end{defi}
We call a matrix $\mathcal{C}\in\kk^{\NN\times \NN}$ regular if for every $k\in\NN$ its upper left $k\times k$ submatrix is regular. \\

In the following lemma we write down a basic property of Hankel matrices over finite fields, that is a consequence of \cite[Corollary~1]{larnie1995} and points out the relevance of the golden ratio analogs. 
\begin{lemma}\label{lem:HM}
We denote the Hankel matrix $\mathcal{M}\in\kk^{\NN\times \NN}$ as $\mathcal{M}=(m_{i,j})_{i,j\geq 1}$. Then the following assertion are equivalent. 
\begin{enumerate}
	\item The Hankel matrix $\mathcal{M}\in\kk^{\NN\times \NN}$ is regular.  
	\item The partial quotients in the continued fraction of $L=\sum_{j=1}^\infty m_{1,j}X^{-j}$ have all lowest possible degree $1$.
\end{enumerate}
\end{lemma}
\begin{proof}
Note that each Laurent series defines a Hankel matrix and vice versa. Now from the definition of the quality parameter $t$ of the Kronecker type sequence, which is used in \cite{larnie1995}, it is easily checked that $t=0$ if and only if the generating Hankel matrix is regular. Now \cite[Corollary~1]{larnie1995}, which states $t=0$ if and only if all continued fraction coefficients of $L$ have degree $1$, completes the proof. 
\end{proof}

In the rest of the paper we aim for explicit formulas for the coefficients in the series expansions of golden ratio analogs, in particular for the specific case where $\kk$ is a finite field $\FF_q$. \\

We start with the most basic example $[0;\overline{X}]$ over $\FF_2$.

\begin{example}\label{ex:mostbasic}{\rm
Choose $\kk$ as the binary field $\FF_2$ and $L=
[0;\overline{X}]\in\FF_2((X^{-1}))$ satisfying $L^2+X L+1=0$. Now because the field $\FF_2((X^{-1}))$ has characteristic $2$ we can make a guess on the series expansion of $L$ and verify it by computing $L^2+X L+1$. Now for $$L:=\sum_{n=1}^\infty X^{-2^n+1},$$
we obtain
$$L^2+X L+1=\sum_{n=1}^\infty X^{-2^{n+1}+2}+\sum_{n=1}^\infty X^{-2^n+2}+1=X^{0}+1=0.$$
Note that in the before last step we made an index shift and in every step we used the characteristic $2$.}
\end{example}

\begin{remark} {\rm
We compare the series $L=\sum_{n=1}^\infty X^{-2^n+1}$ with 2 times Shallit's number $\sigma$ in base $4$, that was mentioned already in Example~\ref{ex:shallit},  $2\sigma=\sum_{k=1}^\infty 2^{-2^k+1}$. Interestingly we have the same formula for the exponents, but with the difference, that the first number is a quadratic irrational over $\FF_2(X)$ but the second is a transcendental real number (see \cite{shallit1}).}
\end{remark}

Note that over a more general field with characteristic greater than $2$ or zero, squaring a series is not that trivial. It is ad hoc not clear how to find the expansion of $[0;\overline{X}]$ over e.g. $\Fp$ where $p$ is an odd prime number. The method developed in this paper is based on studying the associated Hankel matrices. 
As side products we discover plenty of interesting and astonishing results. 
Just to announce two of them. We give the unique $\mathcal{L}_{\phi}\mathcal{U}_{\phi}$ decomposition of the Hankel matrix $\mathcal{M}_{\phi}$ for any golden ratio analog $\phi$ over $\kk$, where the columns of $\mathcal{U}_{\phi}$ and rows of $\mathcal{L}_{\phi}$ follow easy to define recursions. We identify for specific examples of golden ratio analogs $\phi$ nice patterns in $\mathcal{U}_{\phi}$, that will give us a hint on the Laurent series expansion of $\phi$.


\section{The $\mathcal{L}\mathcal{U}$ decompositions of the Hankel matrix associated with a golden ratio analog}

Throughout this section $\kk$ and $\phi\in\Phi\subseteq\overline{\kk((X^{-1}))}$ will be fixed. Thus the continued fraction $[0;A_1(X),A_2(X),A_3(X),\ldots]$ of $\phi$ satisfies $\deg(A_i(X))=1$ for all $i\in\NN$. 

We define the sequence of Fibonacci polynomials $\mathcal{F}_\phi$ associated with $\phi$ given by the denominators of the convergents to $\phi$.

\begin{defi}[Fibonacci polynomials associated with $\phi$]\label{def:FibPol}
For each explicitly given $$\phi=[0;A_1(X),A_2(X),A_3(X),\ldots]\in\Phi$$ over a field $\kk$ the sequence of Fibonacci Polynomials $$\mathcal{F}_{\phi}:=(F_n(X))_{n\geq 0}$$ in $\kk[X]$ is defined recursively by $F_n(X)=A_n(X)F_{n-1}(X)+F_{n-2}(X)$ for $n\in\NN$ and with initial values $F_{-1}(X)=0$ and $ F_0(X)=1$. 
\end{defi}

\begin{remark}{\rm
Often the notion ``Fibonacci polynomials'' is used for $\mathcal{F}_{\phi}$ in the specific case where $\phi$ is $[0;\overline{X}]$.}
\end{remark}

We observe that all Fibonacci polynomials associated with $\phi\in\Phi$ satisfy $\deg(F_n(X))=n$ for all $n\in\NN_0$. 
Hence we are able to define a so-called ``Zeckendorf representation'' in $\kk[X]$. We note that the sequence of Fibonacci numbers $(F_n)_{n\geq 0}$ in $\NN_0$ can be used to represent any non-negative integers $m$ in terms of different Fibonacci numbers, i.e. $m=\sum_{i=1}^\infty \delta_iF_i$ with $\delta_{i}\in\{0,1\}$ and only finitely many $\delta_i=1$. For the Zeckendorf representation we require, if $\delta_i=1$ then $\delta_{i+1}=0$. The latter guaranties the uniqueness of this representation.

Let $\mathcal{F}_{\phi}:=(F_n(X))_{n\geq 0}$ be the sequence of Fibonacci polynomials associated with $\phi\in\Phi$ and $(f_n)_{n\geq 0}$ is the sequence of the leading coefficients in $\kk$, i.e., $f_n=\lc(F_n(X))$.

Let $P(X)\in\kk[X]$ with $\deg(P(X))=r\in\NN_0$. Usually, $P(X)$ is expressed in terms of powers of $X$, i.e. $P(X)=\sum_{i=0}^r p_iX^i$ with $p_i\in\kk$ for $i=0,\ldots,r$ and $p_r\neq 0$. 
By the following greedy algorithm one can represent 
$P(X)$ in terms of the Fibonacci polynomials, i.e.,
$$P(X)=\sum_{i=0}^rz_iF_i(X)$$
with $z_i\in\kk$ and $z_r\neq 0$.
\begin{algo}[$\phi$-Zeckendorf representation]\label{algo:ZecRep}
{\rm Proceed as follows: 
\begin{itemize}
	\item[-] Set $P(X):=P_r(X)$, $z_r=f_r^{-1}p_r$, and $P_{r-1}(X):=P_r(X)-z_rF_r(X)$. 
	\item[-] For given $P_i(X)$ with $i\in\{1,\ldots, r-1\}$ proceed as follows: If $\deg(P_{i})<i$ set $z_{i}=0$ and $P_{i-1}=P_{i}$, else express $P_{i}(X)$ in terms of powers of $X$ and use its leading coefficient $\lc(P_{i}(X))$ to define $z_{i}=f_{i}^{-1}\lc(P_{i}(X))$ and set $P_{i-1}(X):=P_{i}(X)-z_{i}F_{i}(X)$
	\item[-] Finally set $z_0=f_0^{-1}P_0$. 
\end{itemize}}\end{algo}

The following lemma points out an important advantage of the Zeckendorf representation of a polynomial $n(X)$ in $\kk[X]$. 

\begin{lemma}\label{lem:ValZec}
Let $\phi\in\Phi$ and $n(X)\in\kk[X]\setminus\{0\}$ having associated $\phi$-Zeckendorf representation $\sum_{i=0}^rz_iF_i(X)$. Then $$\nu_{\infty}(\{n(X)\phi\})=-j-1,$$ where $j\in\{0,\ldots,r\}$ is minimal such that $z_j\neq 0$. 
\end{lemma}
\begin{proof}
This follows from $\nu_{\infty}(\{F_i(X)\phi\})=-\deg(F_{i+1}(X))$ as $F_i(X)$ is the denominator of the $i$th convergent to $\phi$ and $\deg(F_i(X))=i$ for all $i\in\NN_0$ and $\nu_{\infty}$ is a non-archimedean valuation, that says for $a,b\in\kk((X^{-1}))$ with $\nu_{\infty}(a)\neq \nu_{\infty}(b)$ we have $\nu_{\infty}(a+b)=\max(\nu_{\infty}(a),\nu_{\infty}(b))$. 
\end{proof}

Next we define matrices associated to a golden ratio analog $\phi$ based on its continued fraction coefficients.

\begin{defi}[Matrices associated to $\phi$]\label{def:mat}
We introduce the \emph{tridiagonal matrix} $\mathcal{R}_\phi=(a_{i,j})_{i\geq 1,j\geq 1}$ associated to $\phi=[0;A_1(X),A_2(X),A_3(X),\ldots]$ as the product $\mathcal{B}_{\phi}\mathcal{D}_{\phi}$ of the two matrices $\mathcal{B}_{\phi},\,\mathcal{D}_{\phi}$, that are defined as follows. 

We write $A_i(X)=u_iX+v_i$ with $u_i,v_i\in\kk$ and $u_i\neq 0$ and define $\mathcal{B}_{\phi}=(b_{i,j})_{i\geq 1,j\geq 1}$ over $\kk$ by
$$b_{i,j}=\left\{\begin{array}{ll}
	-v_i&\quad \mbox{if }i=j\\
	1&\quad \mbox{if }i=j+1\\
	-1&\quad \mbox{if }i=j-1\\
	0&\quad \mbox{else.}	\end{array}\right.$$
And $\mathcal{D}_{\phi}$ is the non-singular diagonal matrix
 $\mathcal{D}_{\phi}=\diag(u_1^{-1},u_2^{-1},u_3^{-1},\ldots)$.

Using the tridiagonal matrix $\mathcal{R}_{\phi}$ we recursively construct two matrices $\mathcal{U}_\phi$ and $\mathcal{L}_\phi$. 

The matrix $\mathcal{U}_\phi$ is obtained columnwise: define the first column as $\bsu_1=(1,0,0,0,\ldots)^T$ and set $\bsu_{n+1}=\mathcal{R}_{\phi}\bsu_n$
for $n\in\NN$. 
The matrix $\mathcal{L}_\phi$ is obtained row by row, by setting $\bsl_1=(1,0,0,0,\ldots)$ and $\bsl_{n+1}=\bsl_n\mathcal{R}_{\phi}$ for $n\in\NN$. 
\end{defi}
Note that in the special case where $\kk=\FF_2$, $\mathcal{R}_{\phi}^T=\mathcal{R}_{\phi}$ and as a consequence $\mathcal{L}_\phi^T= \mathcal{U}_\phi$. In the following theorem we exhibit specific properties of $\mathcal{L}_\phi$ and $\mathcal{U}_\phi$. 

\begin{theorem}\label{prop:1} We have
\begin{enumerate}
	\item $\mathcal{L}_\phi$ is a non-singular lower triangular matrix and $\mathcal{U}_\phi$ is a non-singular upper triangular matrix. 
	\item The $n$th column of $\mathcal{U}_\phi$ collects the coefficients of the $\phi$-Zeckendorf representation of $X^{n-1}$, namely, if $X^{n-1}=\sum_{i=0}^{n-1} z_iF_i(X)$ then $$\bsu_n=(z_0,z_1,z_2,\ldots,z_{n-1},0,\ldots)^T.$$
	\item The product $\mathcal{L}_\phi\mathcal{U}_\phi$ is a Hankel matrix.
	\item The product $u_1^{-1}\mathcal{L}_\phi\mathcal{U}_\phi$ gives the generating matrix $\mathcal{M}_\phi$ associated with $\phi$.
\end{enumerate}
\end{theorem}

\begin{proof}
The matrix $\mathcal{R}_{\phi}$ is a tridiagonal matrix with nonzero entries to the left and to the right of the diagonal. Hence it is an easy consequence from the definition of $\mathcal{U}_\phi$ and $\mathcal{L}_\phi$ that both have nonzero diagonal entries and are upper and lower resp. triangular matrices. \\
For the second item and for $n=1$ we use $X^0=F_{0}(X)$. 
Suppose for $n\in\NN$ the $n$th column $\bsu_n$ satisfies $\bsu_n=(z_0,z_1,z_2,\ldots,z_{n-1},0,\ldots)^T$ where $X^{n-1}=\sum_{i=0}^{n-1} z_iF_i(X)$. Then 
\begin{align*}
X^{n}&=X\sum_{i=0}^{n-1} z_iF_i(X)\\
&=\sum_{i=0}^{n-1} z_i u_{i+1}^{-1}\Big((u_{i+1}X+v_{i+1})F_i(X)+F_{i-1}(X)-F_{i-1}(X)-v_{i+1}F_i(X)\Big)\\
&=\sum_{i=0}^{n-1} z_iu_{i+1}^{-1}F_{i+1}(X)-\sum_{i=0}^{n-1} z_iu_{i+1}^{-1}F_{i-1}(X)-\sum_{i=0}^{n-1} z_iu_{i+1}^{-1}v_{i+1}F_i(X)\\
&=\sum_{i=1}^{n} z_{i-1}u_{i}^{-1}F_{i}(X)-\sum_{i=-1}^{n-2} z_{i+1}u_{i+2}^{-1}F_{i}(X)-\sum_{i=0}^{n-1} z_iu_{i+1}^{-1}v_{i+1}F_i(X).
\end{align*}
Hence the coefficient to $F_i(X)$ in the $\phi$-Zeckendorf representation of $X^n$ is $u_{i}^{-1}z_{i-1}-u_{i+1}^{-1}v_{i+1}z_i-u_{i+2}^{-1}z_{i+1}$, this fits exactly the recursive definition of $\mathcal{U}_\phi$ based on $\mathcal{R}_{\phi}$. \\

The Hankel property of $\mathcal{L}_\phi\mathcal{U}_\phi=:(k_{i,j})_{i\geq 1,j\geq 1}$ follows from 
\begin{align*}k_{i,j}&=(1,0,0,0,\ldots)\mathcal{R}_{\phi}^{i-1} \mathcal{R}_{\phi}^{j-1}(1,0,0,0,\ldots )^T\\
&=(1,0,0,0,\ldots)\mathcal{R}_{\phi}^{i-1-s} \mathcal{R}_{\phi}^{j-1+s}(1,0,0,0,\ldots )^T=k_{i-s,j+s}\end{align*}
for all admissible $i,j,s$. \\

It remains to prove the relation with $\mathcal{M}_\phi$ that is by definition a Hankel matrix.
 Thus comparing the first rows is enough. 
 We start with the first coefficient $c_1$ of $$\phi=\sum_{i=1}^\infty c_iX^{-i}.$$
The first convergent to $\phi$ is $\frac{1}{u_1X+v_1}$. The first entry of the first row of $\mathcal{U}_\phi$ is $1$ and the second is $-u_1^{-1}v_1$. 
From the approximation property of the first convergent we derive
$$\nu_{\infty}\left(\phi-\frac{1}{u_1X+v_1}\right)=-3.$$
Hence the first two coefficients of the series expansions of $\phi$ and $\frac{1}{u_1X+v_1}$ coincide. 
Now we know, that 
$$\frac{1}{u_1X+v_1}=u_1^{-1}X^{-1}\frac{1}{1-(-v_1u_1^{-1}X^{-1})}=u_1^{-1}X^{-1}\sum_{i=0}^\infty(-v_1u_1^{-1}X^{-1})^i$$
hence the first coefficient $c_1$ of $\phi$ is $u_1^{-1}$. \\

For $m\in\NN$ we write $X^m$ in Zeckendorf representation $X^{m}=\sum_{i=0}^mz_i^{(m)}F_i(X)$. So the $(m+1)$st entry of the first row of $\mathcal{U}_\phi$ is given by $z_{0}^{(m)}$. \\
For the $(m+1)$st entry of the first row of $\mathcal{M}_\phi$, or $c_{m+1}$ resp., we use that $c_{m+1}\in\kk $ satisfies
$$\nu_{\infty}(\{X^m\phi-c_{m+1}X^{-1}\})<-1.$$
The latter is equivalent to 
$$\nu_{\infty}(\{X^m\phi-c_{m+1}u_1\phi\})<-1$$
since $c_1=u_1^{-1}$. 
From Lemma~\ref{lem:ValZec} we know
$$  \nu_{\infty}(\{(X^m-z_0^{(m)})\phi\})<-1.$$
Therefore $c_{m+1}u_1=z_0^{(m)}$ and we obtain $$c_{m+1}=u_1^{-1}z_{0}^{(m)}.$$

\end{proof}

Theorem \ref{prop:1} shows that the powers of $\mathcal{R}_{\phi}$, whose upper left entries are related to the coefficients of $\phi$, and whose first columns determine the columns of $\mathcal{U}_\phi$, attract attention for our main problem of finding the series expansion of golden ratio analogs. \\

\section{The most basic golden ratio analog}

We start with the special case $\varphi=[0;\overline{X}]\in\kk((X^{-1}))$. 
A well established formula for the Fibonacci Polynomials $\mathcal{F}_{\varphi}=(F_n(X))_{n\geq 0}$ is the following. Here and later on, a coefficient $l\in\NN_0$ denotes $\sum_{i=1}^l $. 
\begin{lemma}\label{lem:2}
For $\mathcal{F}_{\varphi}=(F_n(X))_{n\geq 0}$ and $n\in\NN_0$ we have 
$$F_n(X)=\sum_{m=0}^{\lfloor n/2\rfloor}\binom{n-m}{m}X^{n-2m}.$$
\end{lemma}
\begin{proof}
We check the formula for $n=0$ and $n=1$. Then $n>1$ follows from induction. 
\end{proof}
We also obtain a formula for the coefficients of the Zeckendorf representations of the powers of $X$, that determine the entries in $\mathcal{U}_{\varphi}$. 
\begin{lemma}\label{lem:3}
Let $m\in\NN_0$ and 
 $\sum_{i=0}^mz_i^{(m)}F_i(X)$ be the $\varphi$-Zeckendorf representation of $X^m$ with $\varphi=[0;\overline{X}]$. 
Then $z_m^{(m)}=1$, $z^{(m)}_{m-2k}=(-1)^k\left(\binom{m}{k}-\binom{m}{k-1}\right)\cdot 1$ for $k=1,\ldots,\lfloor m/2\rfloor$ and all other coefficients are $0$.
Therefore, we have $$\mathcal{U}_{\varphi}=\left(\boldsymbol{1}_{2\NN_0}(j-l)(-1)^{\frac{j-l}{2}}\left(\binom{j-1}{\frac{j-l}{2}}-\binom{j-1}{\frac{j-l}{2}-1}\right)\cdot 1\right)_{l\geq 1,j\geq 1},$$
where $\binom{n}{k}=0$ whenever $n<k$ or $k<0$ and $\boldsymbol{1}_{2\NN_0}(x)=1$ for even $x\in\NN_0$, otherwise $\boldsymbol{1}_{2\NN_0}(x)=0$. 
\end{lemma}
\begin{proof}
The formula for $\mathcal{U}_{\varphi}$ is an immediate consequence of the formula for $z^{(m)}_i$ and Theorem~\ref{prop:1} item~2. The statement on $z^{(m)}_i$ follows from induction on $m$.
\end{proof}

\begin{remark}{\rm
Let $\mathcal{P}_{\varphi}$ be the $\NN\times \NN$ matrix over $\kk$, with the $n$th column defined by the coefficients of $F_{n-1}(X)$ in terms of powers of $X$. Then $$\mathcal{P}_{\varphi}=\left(\boldsymbol{1}_{2\NN_0}(l-i)\binom{\frac{l+i}{2}-1}{\frac{l-i}{2}}\right)_{i\geq 1,l\geq 1}$$
Then $\mathcal{P}_{\varphi}\mathcal{U}_{\varphi}= \mathcal{U}_{\varphi}\mathcal{P}_{\varphi}=\mathcal{I}$, the $\NN\times \NN$ identity matrix over $\kk$. 
Hence, we obtain, as a side-product, the two combinatorial identities involving \textit{Catalans triangle numbers} $B_{n,m}$ which we define as $B_{n,m}:=\binom{n-1}{m}-\binom{n-1}{m-1}$, 
$$\sum_{r=0}^k(-1)^{k-r}\binom{i+r}{r}\left(\binom{i+2k}{k-r}-\binom{i+2k}{k-r-1}\right)=\sum_{r=0}^k(-1)^{k-r}\binom{i+r}{r}B_{i+2k+1,k-r}=0$$
and 
$$\sum_{r=0}^k(-1)^r\left(\binom{i+2r}{r}-\binom{i+2r}{r-1}\right)\binom{i+k+r}{k-r}=\sum_{r=0}^k(-1)^rB_{i+2r+1,r}\binom{i+k+r}{k-r}=0,$$
for $k\in\NN$ and $i\in\NN_0$, from 
$$\sum_{l=i}^{j}\boldsymbol{1}_{2\NN_0}(j-l)\boldsymbol{1}_{2\NN_0}(l-i)\binom{\frac{l+i}{2}}{\frac{l-i}{2}}(-1)^{\frac{j-l}{2}}\left(\binom{j}{\frac{j-l}{2}}-\binom{j}{\frac{j-l}{2}-1}\right)\cdot 1=\delta_{i,j}$$
and 
$$\sum_{l=i}^{j}\boldsymbol{1}_{2\NN_0}(l-i)\boldsymbol{1}_{2\NN_0}(j-l)(-1)^{\frac{l-i}{2}}\left(\binom{l}{\frac{l-i}{2}}-\binom{l}{\frac{l-i}{2}-1}\right)\binom{\frac{j+l}{2}}{\frac{j-l}{2}}\cdot 1=\delta_{i,j}.$$
Note that Catalan's triangle numbers are also often defined differently by $C_{n,m}:=\binom{n+m}{m}-\binom{n+m}{m-1}$. Hence $B_{n,m}=C_{n-1-m,m}$. 
}
\end{remark}


In the following we derive explicit formulas of the series expansion of the golden ratio analog $\varphi$.

We start with the basic example $\varphi$ over $\FF_2$ or $\FF_q$ with characteristic $2$. 
From Theorem~\ref{prop:1} item 3 we obtain the series expansion $\sum_{i=1}^\infty c_iX^{-i}$ from the first row of $\mathcal{U}_{\varphi}$ in Lemma \ref{lem:3}. Hence $$c_i=\left(\boldsymbol{1}_{2\NN_0}(i-1)(-1)^{\frac{i-1}{2}}\left(\binom{i-1}{\frac{i-1}{2}}-\binom{i-1}{\frac{i-1}{2}-1}\right)\pmod{2}\right)\cdot 1.$$
Because of $\boldsymbol{1}_{2\NN_0}(i-1)$ we obtain $c_{2n}=0$ for all $n\in\NN$, and for $n\in\NN_0$, $c_{2n+1}=((-1)^n(\binom{2n}{n}-\binom{2n}{n})\pmod{2})\cdot 1=((-1)^n C_n\pmod{2})\cdot 1$, where $C_n$ is the $n$th Catalan number, which can be defined via the Catalan's triangle numbers by $C_n:=B_{2n+1,n}=\binom{2n}{n}-\binom{2n}{n-1}$.

We recall the Lucas Theorem for prime numbers $p$, i.e.
$$\binom{m}{n}\equiv \prod_{i=0}^\infty \binom{m_i}{n_i}\pmod{p}$$
with $m_i,\,n_i$ given via the $p$-adic expansion of $m$ and $n$. \\

We expand the even number $2n$ in base $2$ and obtain 
$$2n=\sum_{j=1}^ra_j2^j,\,\mbox{ and }n=\sum_{j=0}^{r-1}a_{j+1}2^j.$$
Thus 
$$\binom{2n}{n}= \binom{a_r}{0}\prod_{j=1}^r\binom{a_{j-1}}{a_j}\equiv 0\pmod{2}.$$
For $\binom{2n}{n-1}$ we write $2n=\sum_{j=l}^ra_j2^j$ with $a_l=1$, then $n-1$ is 
$$\sum_{j=l}^{r-1}a_{j+1}2^{j}+2^{l-1}-1. $$
Now $$\binom{2n}{n-1}\equiv \binom{a_r}{0}\prod_{j=l}^{r-1}\binom{a_{j}}{a_{j+1}}\prod_{u=0}^{l-2}\binom{0}{1}\pmod{2}.$$
This is only nonzero if $l=1$ and $a_l=a_{l+1}=\ldots=a_{r-1}=a_r=1$. So $c_{2n+1}=1$ iff $2n=2^{r+1}-2$ or $2n+1=2^{r+1}-1$ with $r\in\NN$. Adding the index $i=1=2\cdot 0+1=2^{1}-1$ we obtain again the series given in Example \ref{ex:mostbasic}
$$L=\sum_{n=1}^\infty X^{-(2^n-1)}$$ 
in characteristic $2$. \\

For a general finite field with characteristic $p$ other than $2$, this simple method above, does not work, as the $p$-adic expansions of $2n$ and $n$ are not that simply related. For the general case we at least obtain the formula for the series expansion of $\varphi$ involving the Catalan numbers. 
\begin{prop}\label{prop:Catalan}
For $\varphi=[0;\overline{X}]=\sum_{i=1}^\infty c_iX^{-i}$ over $\FF_q$ with characteristic $p$ we have $c_{2r}=0$  for $r\in\NN$. And for $r\in\NN_0$ we have
$$c_{2r+1}=\big((-1)^r\left(\binom{2r}{r}-\binom{2r}{r-1}\right)\pmod{p}\big)\cdot 1=\big((-1)^r C_r\pmod{p} \big)\cdot 1$$
where $C_r$ is the $r$th Catalan number, satisfying $C_r=\binom{2r}{r}-\binom{2r}{r-1}=\frac{1}{r+1} \binom{2r}{r}=\sum_{j=0}^r\binom{r}{j}^2=\prod_{k=2}^r\frac{r+k}{k}$. 

\end{prop}

\begin{remark}{\rm
Proposition~\ref{prop:Catalan} and $\varphi^2+\varphi X-1=0$ imply the well-known recurrence relation for the Catalan numbers, 
$$\sum_{i=0}^nC_{i}C_{n-i}=C_{n+1}.$$}
\end{remark}

The distribution of the Catalan numbers modulo $p$ follows specific regularities that are worked out above for modulus $2$. In moduli $2,\,3,\,5$, and $7$, these regularities are given in Figure~\ref{fig:cat_modp}. One aim of this paper is a nice formula for the series expansion of $\varphi$ over a finite field, which we found already for characteristic $2$. 
\begin{figure}[h]
	\centering
		\includegraphics[width=\textwidth]{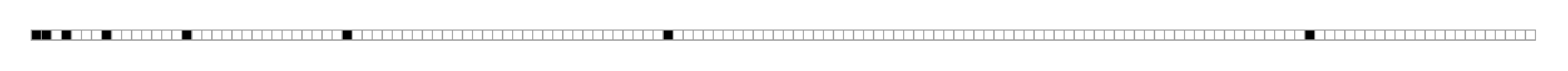}
		\includegraphics[width=\textwidth]{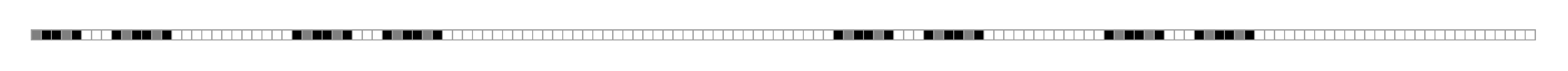}
		\includegraphics[width=\textwidth]{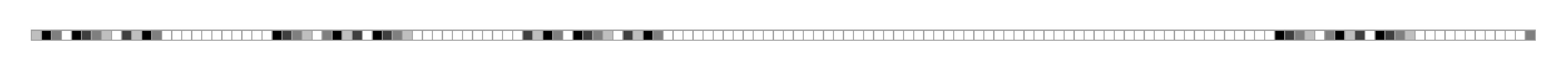}
		\includegraphics[width=\textwidth]{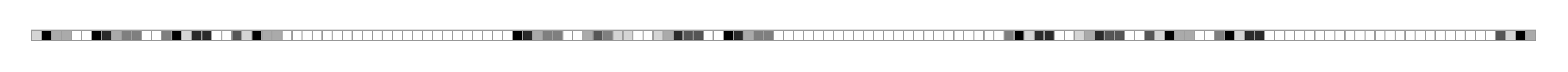}
	\caption{The sequence $((-1)^nC_n)_{n\geq 0}$ modulo $2$, $3$, $5$, and $7$.}
	\label{fig:cat_modp}
\end{figure}

One way to describe the pattern modulo $p$ is to exploit the Lucas Theorem together with the base $p$ expansion of $2n$ and $n$. An alternative approach is the investigation of a certain fractal structure in $\mathcal{U}_\varphi$ via searching for pattern in $\mathcal{R}_{\varphi}^l$. An advantage of the second approach is, that we can do this for more general $\phi\in\Phi$.\\ 
Let us consider the first approach before we work on the second one in Section~\ref{sec:4}. Obviously, the question whether multiplication of $n$ with $2$ causes at least one carry in the $p$-adic expansion of $n$ or not is crucial for the first approach. 
\begin{lemma}\label{lem:4}
Let $p$ be an odd prime number and $n\in\NN_0$ with unique $p$-adic expansion $n=\sum_{i=0}^\infty a_i p^i$ where $a_i\in\{0,1,\ldots,p-1\}$. \\

For the case where $a_0\neq p-1$, we have 
$$C_n \pmod{p}\neq0 \quad \mbox{ if and only if }\quad a_i\in\{0,1,\ldots,\lfloor(p-1)/2\rfloor\}$$ for all $i\in\NN_0$. \\

For $a_0=p-1$, let $l\in\NN$ be maximal such that $a_0=a_1=\ldots=a_{l-1}=p-1$. Then with $n_l:=\sum_{i=0}^\infty a_{i+l} p^i$, we have $$C_n\equiv (-1)(2n_l+1)C_{n_l}\pmod{p}.$$
\end{lemma}

\begin{proof}
We see that for any $m\in\NN\setminus\{1\}$ 
\begin{eqnarray}\label{eq:casedist}
2a \pmod{m} =\left\{\begin{array}{ll} 
2a \geq a&\mbox{ if }a\in\{0,1,\ldots,\lfloor (m-1)/2\rfloor\}\\
2a-m<a&\mbox{ if }a\in\{\lceil m/2\rceil,\ldots,m-1\}.
\end{array}\right.
\end{eqnarray}
The case of $a_0\neq p-1$ follows directly from the Lucas Theorem together with \eqref{eq:casedist}, the trivial fact that $n+1\neq 0\pmod{p}$, and $C_n=\frac{1}{n+1}\binom{2n}{n}$.\\
For $a_0=p-1$ we have $\binom{2n}{n}\equiv 0\pmod{p}$ and we focus on $-\binom{2n}{n-1}$ since $C_n=\binom{2n}{n}-\binom{2n}{n-1}$. Note that $$\binom{2n}{n-1}=\binom{2pn_1+p+p-2}{pn_1+p-2}\equiv \binom{2n_1+1}{n_1} \pmod{p}$$ and that $\binom{2n_1+1}{n_1}=(2n_1+1) C_{n_1}$. 
Hence $C_n\equiv (-1)(2n_1+1)C_{n_1}\pmod{p}$. 
\\
For $a_1= p-1$ use $C_{n_1} \equiv- (2n_2+1)C_{n_2}\pmod{p}$ with $n_2=\sum_{i=0}^\infty a_{i+2} p^i$, and $(2n_1+1)\equiv -1 \pmod{p}$. This step is repeated till we use $C_{n_{l-1}}\equiv -(2n_l+1) C_{n_l} \pmod{p}$.  
\end{proof}

This can be used to compute the series expansion of $\varphi$ for different odd prime characteristics. 
\begin{example}\label{examp:p3}{\rm Let $p=3$. We see that we can build the series expansion $\varphi=\sum_{i=0}^\infty a_iX^{-2i-1}$ by starting with $a_0=1,\,a_1=2$. Then use Step $1$, followed by Step $2$ etc. \\
We describe ``Step $k$'' with $k\in\NN$: For $i=3^k-1$ set $a_i=2$ and repeat the first $3^k-1$ values of the coefficients $a_0,\ldots,a_{3^k-2}$, those are followed up by $3^k$ zeroes. Start Step $k+1$.\\
The values $a_0=1,\,a_1=2$ are easily checked. Then $a_{3^k-1}=2$ follows from $C_{3^k-1}=(-1)C_0$, $C_0=1$, and $(-1)^{3^k-1}=1 $.  Repeating the coefficients $a_0,\ldots,a_{3^k-2}$ follows from the fact that for $3^k+n$ with $n=3^{k-1}+r$ or $n=r$ for $r\in\{0,1,\ldots,3^{k-1}-1\}$ we have $C_{3^k+n}\equiv 2\frac{n+1}{3^k+n+1}C_n\pmod{3}$ and $(-1)^{3^k}=(-1)$. 
The block of following $3^k$ zeroes is a consequence of the fact that here at least one digit in the base $3$ expansion of $n_l$ is $2$.}
\end{example}

Similarly one could determine a pattern in the Laurent series expansion for $\varphi=[0;\overline{X}]$ for any prime number characteristic $p\geq 5$.

\section{Fractal structures in $\mathcal{U}$ for specific golden ratio analogs showing patterns for their series expansions}\label{sec:4}

We start with the definition of operations on matrices that will be frequently used in this section. 

Let $\mathcal{I}$ denote the identity matrix in $ \kk^{\NN\times\NN}$. For $u\in\NN$ and $\mathcal{M}\in\kk^{\NN\times\NN}$, let $[\mathcal{M}]_{u}$ denote the upper left $u\times u$ submatrix of $\mathcal{M}$. 

We introduce the operators $$\oplus,\,\ominus:\kk^{\NN\times\NN}\to \kk^{\NN\times\NN},$$ for $\mathcal{M}=(m_{i,j})_{i\geq 1,j\geq 1}$ by setting $$(\oplus(M))_{i,j}=m_{i-1,j}\mbox{ if } i>1\mbox{ and } 0 \mbox{ else },$$
 and 
$$(\ominus(M))_{i,j}=m_{i+1,j}.$$ 
For finite $l\times l$ matrices $[\mathcal{M}]_l$, $\oplus([\mathcal{M}]_l)$ is defined equally but $\ominus([\mathcal{M}]_l)_{i,j}:=m_{i+1,j}$ if $i\in\{1,2,\ldots,l-1\}$ and $0$ if $i=l$.

Note that $\oplus(\mathcal{I})\cdot \oplus(\mathcal{I})=\oplus(\oplus(\mathcal{I}))$ and $\ominus(\mathcal{I})\cdot \ominus(\mathcal{I})=\ominus(\ominus(\mathcal{I}))$. 

We set $$\oplus^{-k}(M):=\ominus^k(M)\mbox{ and } \ominus^{-k}(M)=\oplus^k(M).$$ 
We define the entry in the $i$th row and $j$th column of the $l$th antidiagonal Matrix $\mathcal{J}_{l}\in\FF_q^{\NN\times\NN}$ 
as $1$ if $i+j=l+1$ and $0$ else. \\

We define the entries of the $l$th alternating antidiagonal matrix $\mathcal{J}^{(a)}_{l}\in\FF_q^{\NN\times\NN}$ 
as $(-1)^{i-1}$ if $i+j=l+1$ and $0$ else. 

\subsection{The case of characteristic $2$}

We start again with the special case of characteristic $2$. Then we obtain for $\phi=[0;X+v_1,X+v_2,\ldots]$ 
$$\mathcal{R}_{\phi}=\diag(v_1,v_2,\ldots)+\oplus(\mathcal{I})+\ominus({\mathcal{I}}).$$


\begin{theorem}\label{prop:char:2}
Let $\mathcal{U}_\varphi$ be the matrix determined by $\varphi=[0;\overline{X}]$ and $\mathcal{U}_{\overline{\varphi}}$ the one determined by $\overline{\varphi}=[0;\overline{X+1}]$ in characteristic $2$. 
We have in both matrices a fractal structure, namely, 
$$[\mathcal{U}_{\varphi}]_{2^k}=\begin{pmatrix} [\mathcal{U}_{\varphi}]_{2^{k-1}}&\ominus([\mathcal{J}_{2^{k-1}}]_{2^{k-1}})\cdot[\mathcal{U}_{\varphi}]_{2^{k-1}}\\0&[\mathcal{U}_{\varphi}]_{2^{k-1}}\end{pmatrix}$$
and 
$$[\mathcal{U}_{\overline{\varphi}}]_{2^k}=\begin{pmatrix} [\mathcal{U}_{\overline{\varphi}}]_{2^{k-1}}&([\mathcal{I}]_{2^{k-1}}+\ominus([\mathcal{J}_{2^{k-1}}]_{2^{k-1}}))[\mathcal{U}_{\overline{\varphi}}]_{2^{k-1}}\\0&[\mathcal{U}_{\overline{\varphi}}]_{2^{k-1}}\end{pmatrix}$$
for $k>1$ with the initial values $[\mathcal{U}_{\varphi}]_2=\begin{pmatrix}1&0\\0&1\end{pmatrix}$ and $[\mathcal{U}_{\overline{\varphi}}]_2=\begin{pmatrix}1&1\\0&1\end{pmatrix}$. 
\end{theorem}
\begin{proof}
Note that from the definition of $\mathcal{U}_\phi$, the right blocks in $[\mathcal{U}_\phi]_{2^{k}}$ are determined by the left via
$$\begin{pmatrix} ([\mathcal{U}_{\phi}]_{2^{k-1}}\\0\end{pmatrix}\mapsto [\mathcal{R}_{\phi}]_{2^k}^{2^{k-1}}\cdot\begin{pmatrix} [\mathcal{U}_{\phi}]_{2^{k-1}}\\0\end{pmatrix}.$$

The second structure in $\mathcal{U}_{\overline{\varphi}}$ follows from the one in $\mathcal{U}_{{\varphi}}$ together with the fact that $([\mathcal{R}_{\varphi}]_{2^k}+[\mathcal{I}]_{2^k})^2=[\mathcal{R}_{\varphi}]_{2^k}^2+[\mathcal{I}]_{2^k}$ in characteristic $2$. 

In the following we concentrate on $\mathcal{R}_{\varphi}$ and we prove the specific of form
$$[\mathcal{R}_{\varphi}]_{2^k}^{2^{k-1}}=\begin{pmatrix} \ominus([\mathcal{J}_{2^{k-1}}]_{2^{k-1}}) & [\mathcal{I}]_{2^{k-1}} \\ [\mathcal{I}]_{2^{k-1}} & \oplus([\mathcal{J}_{2^{k-1}}]_{2^{k-1}})\end{pmatrix}.$$

Using $[\mathcal{R}_{\varphi}]_{2^k}=[\oplus(\mathcal{I})]_{2^k}+[\ominus(\mathcal{I})]_{2^k}$ we observe 
$$([\oplus(\mathcal{I})]_{2^k}+[\ominus(\mathcal{I})]_{2^k})^{2}=[\oplus^2(\mathcal{I})]_{2^k}+[\ominus^2(\mathcal{I})]_{2^k}+[\oplus(\mathcal{I})]_{2^k}\cdot[\ominus(\mathcal{I})]_{2^k}+[\ominus(\mathcal{I})]_{2^k}\cdot[\oplus(\mathcal{I})]_{2^k}$$
and $$[\oplus(\mathcal{I})]_{2^k}\cdot[\ominus(\mathcal{I})]_{2^k}+[\ominus(\mathcal{I})]_{2^k}\cdot[\oplus(\mathcal{I})]_{2^k}=\begin{pmatrix}1&0&\ldots&0& 0\\0&0&\ldots&0& 0\\\vdots &\vdots & & \vdots&\vdots \\0&0&\ldots&0& 0\\0&0&\ldots&0& 1\end{pmatrix}=:Q_2.$$
Then
$$([\oplus^2(\mathcal{I})]_{2^k}+[\ominus^2(\mathcal{I})]_{2^k}+Q_2)^{2}$$
$$=
[\oplus^{2^2}(\mathcal{I})]_{2^k}+[\ominus^{2^2}(\mathcal{I})]_{2^k}+$$
$$+[\oplus^2(\mathcal{I})]_{2^k}Q_2+Q_2[\ominus^2(\mathcal{I})]_{2^k}
+[\ominus^2(\mathcal{I})]_{2^k}Q_2+Q_2[\oplus^2(\mathcal{I})]_{2^k}
+[\oplus^2(\mathcal{I})]_{2^k}[\ominus^2(\mathcal{I})]_{2^k}+[\ominus^2(\mathcal{I})]_{2^k}[\oplus^2(\mathcal{I})]_{2^k}+Q_2^2$$
and 
$$[\oplus^2(\mathcal{I})]_{2^k}Q_2+Q_2[\ominus^2(\mathcal{I})]_{2^k}
+[\ominus^2(\mathcal{I})]_{2^k}Q_2+Q_2[\oplus^2(\mathcal{I})]_{2^k}
+[\oplus^2(\mathcal{I})]_{2^k}[\ominus^2(\mathcal{I})]_{2^k}+[\ominus^2(\mathcal{I})]_{2^k}[\oplus^2(\mathcal{I})]_{2^k}+Q_2^2$$
$$\quad\quad \quad\quad\quad \quad\quad\quad \quad=\begin{pmatrix}0&0&1&\ldots&0& 0&0\\0&1&0&\ldots&0& 0&0\\1&0&0&\ldots&0& 0&0\\ \vdots&\vdots &\vdots && \vdots&\vdots &\vdots\\0&0&0&\ldots&0& 0&1\\0&0&0&\ldots&0& 1&0\\0&0&0&\ldots&1&0& 0\end{pmatrix}=:Q_{2^2}.$$
By induction on $l$ we derive that $Q_{2^l}$ for $l<k$ is the zero matrix except of its upper-left and lower-right $2^l-1\times 2^l-1$ submatrices that are antidiagonal matrices. 
\end{proof}

\begin{remark}{\rm
Note that the fractal structure in $\mathcal{U}_\varphi$ given in the lemma above once again proves $$\varphi=[0;\overline{X}]=\sum_{n=1}^\infty X^{-2^n+1}.$$ 
From the fractal structure of $\mathcal{U}_{\overline{\varphi}}$ one may derive $$\overline{\varphi}=[0;\overline{X+1}]=\sum_{n=1}^\infty \nu_2(2n)(X^{-2n+1}+X^{-2n})$$ where $\nu_2(n):=\max\{l\in\NN_0:2^l|n\}$. \\
Similarly, one may derive a formula for the Laurent series expansion of $\phi=[0;\overline{X,X+1}]$ or $[0;\overline{X+1,X}]$.  
}\end{remark}

\subsection{The case of characteristic $p>2$}

In the last part of this section, we aim for a generalization of the above lemma for odd characteristic $p$ and $\phi_{u,v}=[0;\overline{uX+v}]$ and $$\mathcal{R}_{\phi_{u,v}}=u^{-1}(\oplus{\mathcal{I}}-\ominus{\mathcal{I}}-v \mathcal{I}).$$ Then the special case $[0;\overline{X}]$ is given by $\phi_{1,0}$ and
  $\mathcal{R}_{\phi_{1,0}}=(\oplus{\mathcal{I}}-\ominus{\mathcal{I}})$. 

In the following we search for a pattern in $[\mathcal{R}_{\phi_{u,v}}]^{p^{k-1}l}_{p^k}$, with $l\in\{1,\ldots,p-1\}$ and $k\in\NN$ that together with Theorem~\ref{prop:1} gives insight to the Laurent series expansion of $\phi_{u,v}$. The main difficulty is that $\ominus(\mathcal{I})$ and $\oplus(\mathcal{I})$ do not commute, as 
$$\mathcal{I}=\ominus(\mathcal{I})\oplus(\mathcal{I})\neq \oplus(\mathcal{I})\cdot \ominus(\mathcal{I})=:\partial^{(1)}(\mathcal{I})$$
where $\partial^{(n)}:\FF_q^{\NN\times\NN}\to\FF_q^{\NN\times\NN}$ with $n\in\NN$ denotes the $n$th \textit{flattening-operator}, that sets all nonzero entries in the first $n$ rows in a matrix in $\FF_q^{\NN\times\NN}$ equal to zero. 
 
We begin with binomial type theorem for $\left(\oplus(\mathcal{I})-\ominus(\mathcal{I})\right)^{m}$. 
\begin{theorem}\label{prop:3}
 For every $m\in\NN$ we have
$$\Big(\oplus(\mathcal{I})-\ominus(\mathcal{I})\Big)^{m}=\sum_{r=0}^m\binom{m}{r}(-1)^r\oplus^{(m-2r)}(\mathcal{I}) \, +\,(-1)^m\sum_{l=0}^{\lfloor m/2\rfloor-1}\binom{m}{l}(-1)^l\mathcal{J}^{(a)}_{m-1-2l}. $$
\end{theorem}
\begin{proof}
We check the formula for $m=1$. 
For $m>1$ we use induction on $m$. 
We observe for $k\in\NN$ three equalities that will help for the proof 
\begin{align*}
\oplus^k(\mathcal{I})\cdot \ominus(\mathcal{I})=&\oplus^{k-1}(\mathcal{I})-\oplus^{k-1}(\mathcal{I})\cdot\mathcal{J}^{(a)}_{1}\\
\mathcal{J}^{(a)}_{k}\cdot \oplus(\mathcal{I})=&\mathcal{J}^{(a)}_{k-1}\\
\mathcal{J}^{(a)}_{k}\cdot (-1)\ominus(\mathcal{I})=&(-1)\mathcal{J}^{(a)}_{k+1}+(-1)^k\oplus^{k}(\mathcal{I})\cdot\mathcal{J}^{(a)}_{1}
\end{align*}
Let $m>1$. Then we use the induction hypothesis for $m-1$. 
\begin{eqnarray*}
\lefteqn{\Big(\oplus(\mathcal{I})-\ominus(\mathcal{I})\Big)^{m}}\\
&=&\left(\oplus(\mathcal{I})-\ominus(\mathcal{I})\right)^{m-1}\left(\oplus(\mathcal{I})-\ominus(\mathcal{I})\right)\\
&=&\left(\sum_{r=0}^{m-1}\binom{m-1}{r}(-1)^r\oplus^{(m-1-2(r))}(\mathcal{I}) \, +\,(-1)^{m-1}\sum_{l=0}^{\lfloor (m-1)/2\rfloor-1}\binom{m-1}{l}(-1)^{l}\mathcal{J}^{(a)}_{m-2-2(l)}   \right)\\
&&\quad\cdot \Big(\oplus(\mathcal{I})-\ominus(\mathcal{I})\Big).
\end{eqnarray*}
It is easily checked that 
$$\sum_{r=0}^{m-1}\binom{m-1}{r}(-1)^r\oplus^{(m-1-2r)}(\mathcal{I})\cdot\Big(\oplus(\mathcal{I})-\ominus(\mathcal{I})\Big)$$
equals 
$$ \sum_{r=0}^m\binom{m}{r}(-1)^r\oplus^{(m-2r)}(\mathcal{I})+\sum_{r=0}^{\lfloor (m-1)/2\rfloor-1}\binom{m-1}{r}(-1)^r\oplus^{m-2(r+1)}(\mathcal{I})\mathcal{J}^{(a)}_1.$$
It remains to rewrite 
$$(-1)^{m-1}\sum_{l=0}^{\lfloor (m-1)/2\rfloor-1}\binom{m-1}{l}(-1)^{l}\mathcal{J}^{(a)}_{m-2-2l} \big(\oplus(\mathcal{I})-\ominus(\mathcal{I})\big)$$
in the form
$$(-1)^{m-1}\sum_{l=0}^{\lfloor (m-1)/2\rfloor-1}\binom{m-1}{l}(-1)^{l}\mathcal{J}^{(a)}_{m-1-2(l+1)}+(-1)^{m-1}\sum_{l=0}^{\lfloor (m-1)/2\rfloor-1}\binom{m-1}{l}(-1)^{l+1}\mathcal{J}^{(a)}_{m-1-2l}$$
$$+(-1)^{m-1}\sum_{l=0}^{\lfloor (m-1)/2\rfloor-1}\binom{m-1}{l}(-1)^l(-1)^{m-2-2l}\oplus^{m-2(l+1)}(\mathcal{I})\mathcal{J}^{(a)}_1.$$
Merging the first two sums and adding the result above yields the desired equality. 


\end{proof}

From Theorem~\ref{prop:3} and the Lucas Theorem  we derive the following corollary for finite fields. 
\begin{coro}\label{coro:1}
Let $p>2$ be the characteristic of $\FF_q$. For $k\in\NN$ we have
$$\left(\oplus(\mathcal{I})-\ominus(\mathcal{I})\right)^{p^k}=\oplus^{p^k}(\mathcal{I})-\ominus^{p^k}(\mathcal{I})-\mathcal{J}^{(a)}_{p^k-1}.$$
Furthermore, for even $l\in\{1,\ldots,p-1\}$ we have

$$\left(\oplus(\mathcal{I})-\ominus(\mathcal{I})\right)^{p^kl}=\sum_{i=0}^{l}\binom{l}{i}(-1)^i\oplus^{p^k(l-2 i)}(\mathcal{I})+\sum_{i=0}^{l/2-1}\binom{l}{i}(-1)^i\mathcal{J}^{(a)}_{p^k(l-2i)-1}.$$
For odd $l\in\{1,\ldots,p-1\}$, we have
$$\left(\oplus(\mathcal{I})-\ominus(\mathcal{I})\right)^{p^k l}=\sum_{i=0}^{l}\binom{l}{i}(-1)^i\oplus^{p^k(l-2 i)}(\mathcal{I})-\sum_{i=0}^{(l-1)/2-1}\binom{l}{i}(-1)^i\mathcal{J}^{(a)}_{p^k(l-2i)-1}.$$
\end{coro}

From the fact that for any $a\in\FF^{\ast}_p$ we have $a^{p^k-1}=1$, the fact that $$\mathcal{R}_{\phi_{u,v}}=u^{-1}\left(\underbrace{\oplus(\mathcal{I})-\ominus(\mathcal{I})}_{\mathcal{R}_{\phi_{1,0}}}-v\mathcal{I}\right)$$ where $-v\mathcal{I}$ and $\mathcal{R}_{\phi_{1,0}}$ commute, we might derive a generalization of Corollary \ref{coro:1} above. 

\begin{coro}\label{coro:2}
Set $\FF_p$ with characteristic $p>2$. For $k\in\NN$ we have
$$\mathcal{R}_{\phi_{u,v}}^{p^k}=u^{-1}\mathcal{R}_{{\phi}_{1,0}}^{p^{k}}+u^{-1}v\mathcal{I}$$
and 
for $l\in\{1,\ldots,p-1\}$ we have 
$$\mathcal{R}_{\phi_{u,v}}^{p^kl}=u^{-l}\sum_{i=0}^{l}\binom{l}{i}u^{-i}v^{i}\mathcal{R}_{{\phi}_{1,0}}^{p^{k}l-i}$$
over $\FF_p$. 
\end{coro} 


We now exemplary apply the above Corollaries \ref{coro:1} and \ref{coro:2} to the question for a recursive structure in $[\mathcal{U}_{\phi_{u,v}}]_{p^k}$ that is a consequence of the specific forms of powers of $\mathcal{R}_{\phi_{u,v}}$ and that explains the recursive structure of the Laurent series expansion of $\phi_{u,v}=[0;\overline{uX+v}]$. \\
Note that Definition \ref{def:mat} gives us a specific dependence between $[\mathcal{U}_{\phi_{u,v}}]_{p^{k-1}}$ and $[\mathcal{U}_{\phi_{u,v}}]_{p^k}$ (confer also proof of Theorem \ref{prop:char:2}). The binomial type Theorem~\ref{prop:3} generalizes the formula for $[\mathcal{R}_{\varphi}]_{2^k}^{2^{k-1}}$ which is given in the proof of Theorem \ref{prop:char:2}. \\

\begin{example}[$p=3$]{\rm
We regard $\phi_{1,0}$ over $\FF_3$. Then Figure~\ref{fig:p3} shows $[\mathcal{U}_{\phi_{1,0}}]_{3^4}$, $[\mathcal{U}_{\phi_{1,0}}]_{3^3}$, and $[\mathcal{R}_{\phi_{1,0}}^{1\cdot 3^2}]_{3^3}$ as well as $[\mathcal{R}_{\phi_{1,0}}^{2\cdot 3^2}]_{3^3}$. Those matrices show up the self similar structure in $\mathcal{U}_{\phi_{1,0}}$. The first $9$ columns of $[\mathcal{U}_{\phi_{1,0}}]_{3^3}$ multiplied from the left with $[\mathcal{R}_{\phi_{1,0}}^{1\cdot 3^2}]_{3^3}$ give the next $9$ columns. The first $9$ columns of $[\mathcal{U}_{\phi_{1,0}}]_{3^3}$ multiplied from the left with $[\mathcal{R}_{\phi_{1,0}}^{2\cdot 3^2}]_{3^3}$ give the last $9$ columns of $[\mathcal{U}_{\phi_{1,0}}]_{3^3}$. 

\begin{figure}[p]\centering
		\includegraphics[width=0.90\textwidth]{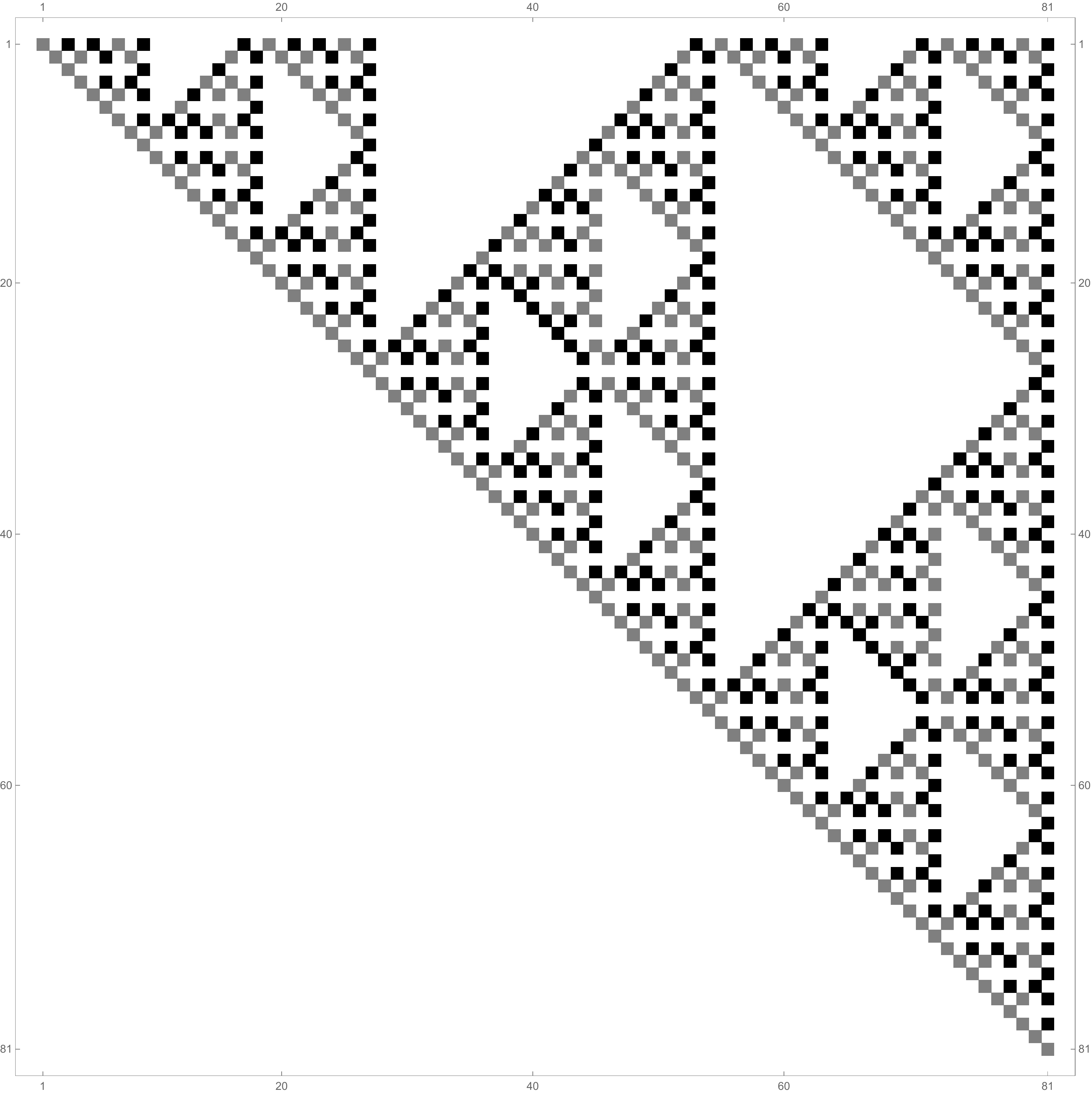}\\
		\includegraphics[width=0.31\textwidth]{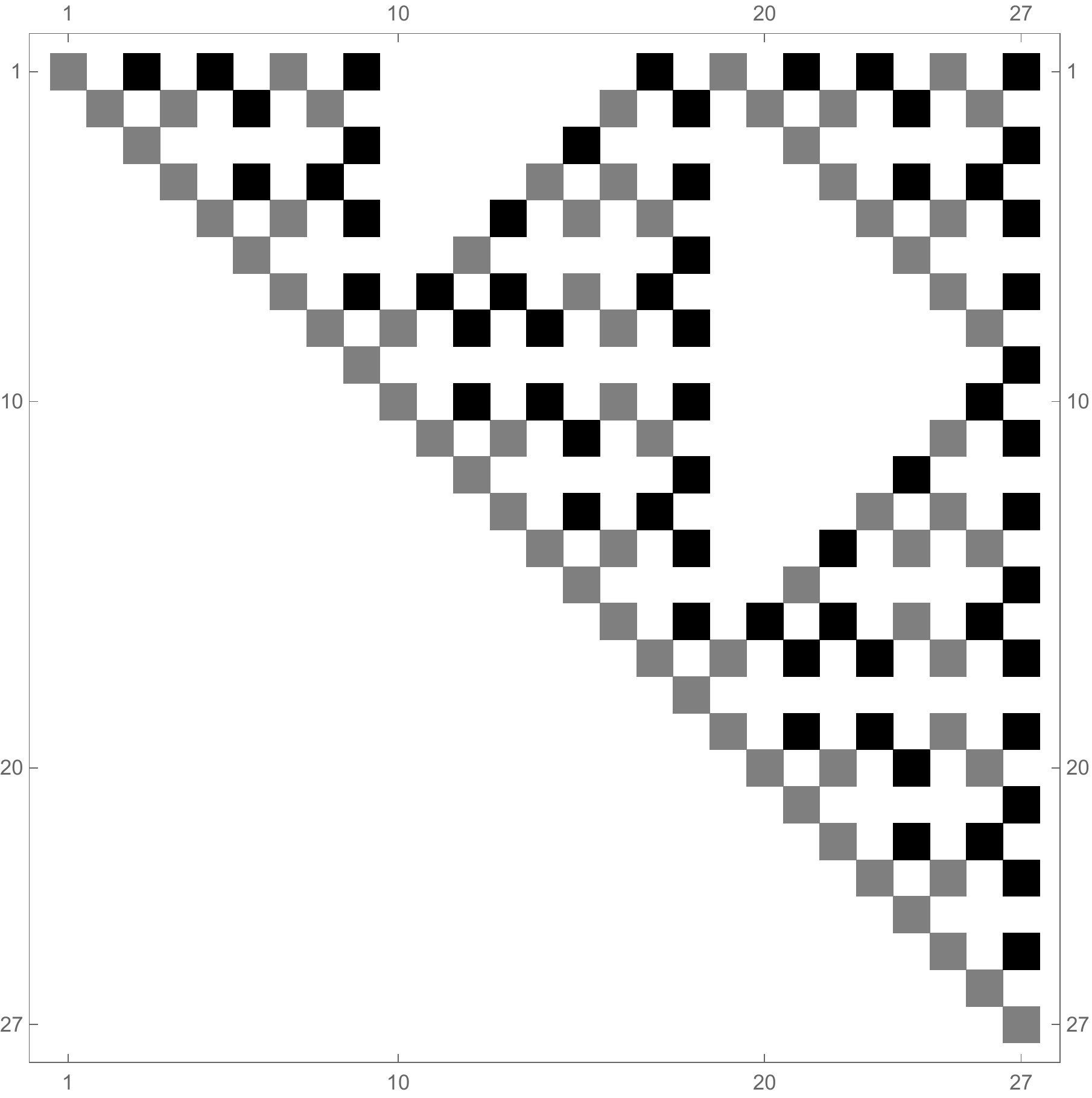} 
		\includegraphics[width=0.31\textwidth]{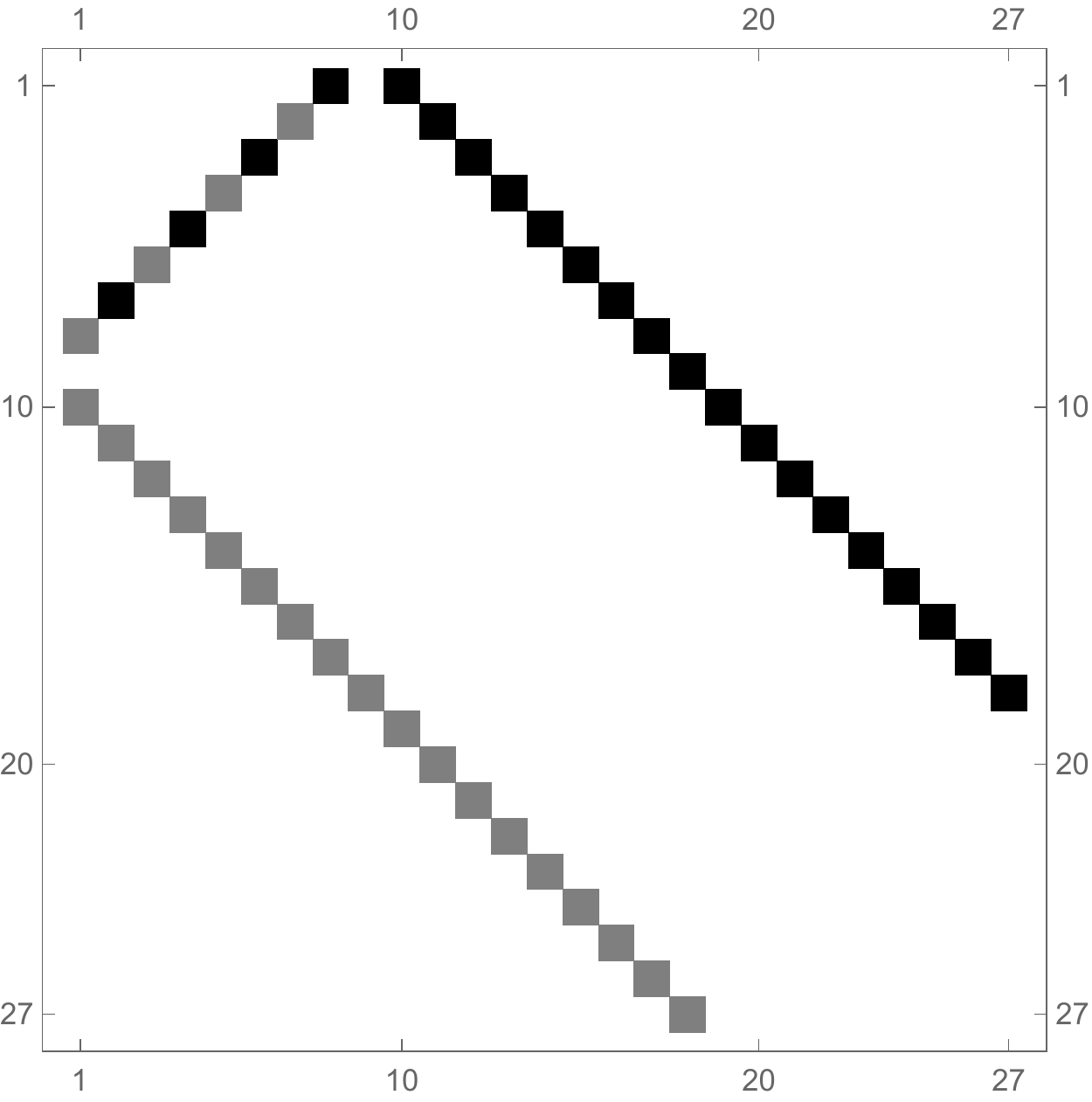} 
		\includegraphics[width=0.31\textwidth]{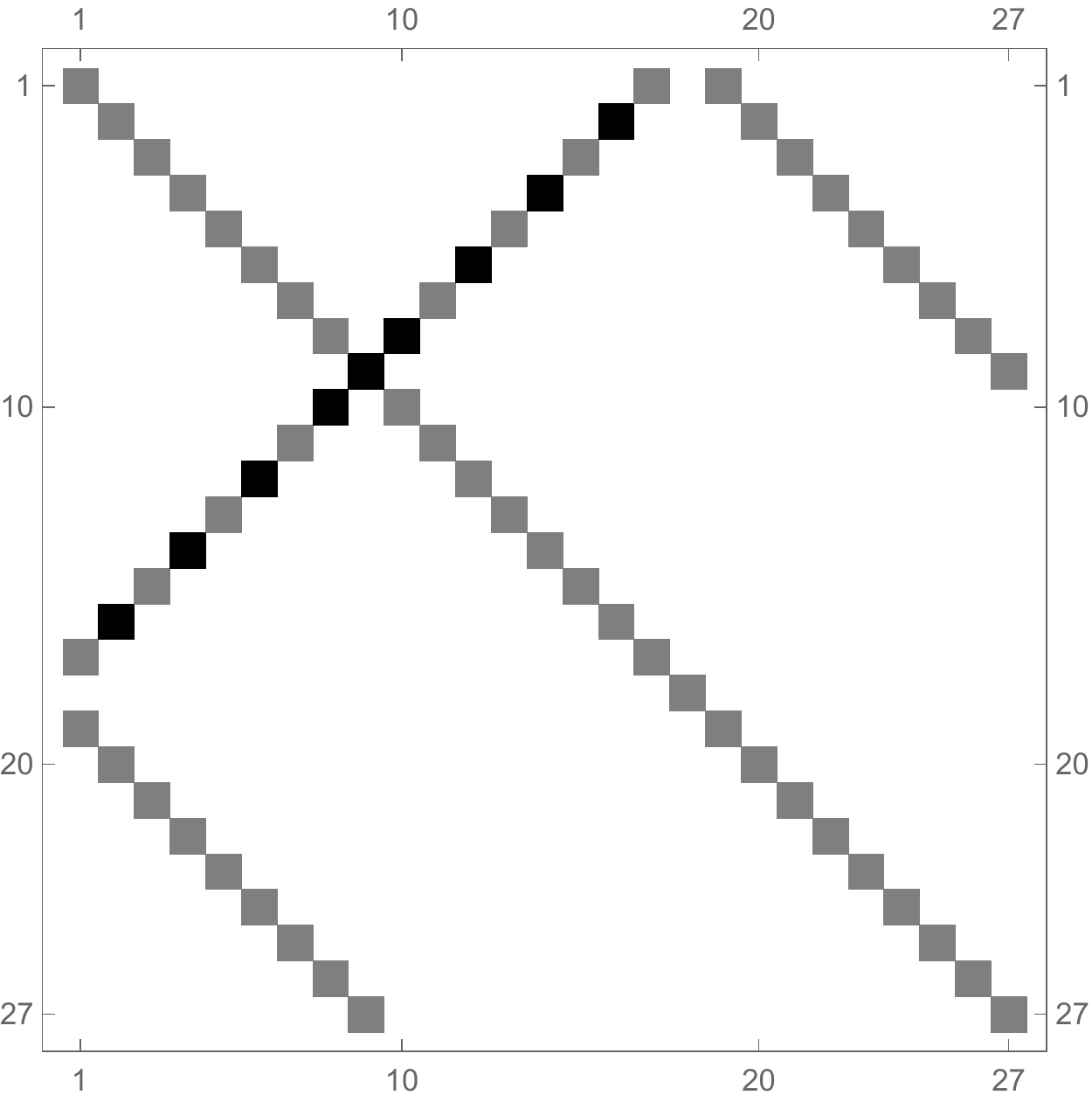} 
	\label{fig:p3}\caption{$[\mathcal{U}_{\phi_{1,0}}]_{3^4},\,[\mathcal{U}_{\phi_{1,0}}]_{3^3},\,[\mathcal{R}_{\phi_{1,0}}^{1\cdot 3^2}]_{3^3}$, and $ [\mathcal{R}_{\phi_{1,0}}^{2\cdot 3^3}]_{3^3}$.}
\end{figure}
Moreover, the self similar structure re-explains the series expansion of $\phi_{1,0}=\sum_{r=0}^\infty c_rX^{-r-1}$ over $\FF_3$ described in Example~\ref{examp:p3}. The values $(c_0,c_1,c_2)=(1,0,2)$ are given in the first row of $[\mathcal{U}_{\phi_{1,0}}]_{3}$. The next string of length $3$ is $(0,2,0)$ the first row of $-[\mathcal{J}^{(a)}_{3-1}]_3\cdot[\mathcal{U}_{\phi_{1,0}}]_{3}$ as $-[\mathcal{J}^{(a)}_{3-1}]_3$ is the upper left $3\times 3$ part of $ [\mathcal{R}_{\phi_{1,0}}^{1\cdot 3^2}]_{3^2}$. The next string of length $3$ is $(1,0,2)$ the first row of $ [\mathcal{I}]_3 \cdot[\mathcal{U}_{\phi_{1,0}}]_{3}$. The next string of length $9$ is then given by the first row of $-[\mathcal{J}^{(a)}_{9-1}]_9 \cdot [\mathcal{U}_{\phi_{1,0}}]_{9}$, i.e. $(0,0,0,0,0,0,0,2,0)$, etc. }
\end{example}

\begin{example}[$p=5$]{\rm
Figure \ref{fig:p5} gives the relevant matrices over $\FF_5$ that give rise to the stepwise series expansion of $\phi_{1,0}=\sum_{r=0}^\infty c_rX^{-r-1}$ over $\FF_5$. 
\begin{figure}[p]
\centering
		\includegraphics[width=0.90\textwidth]{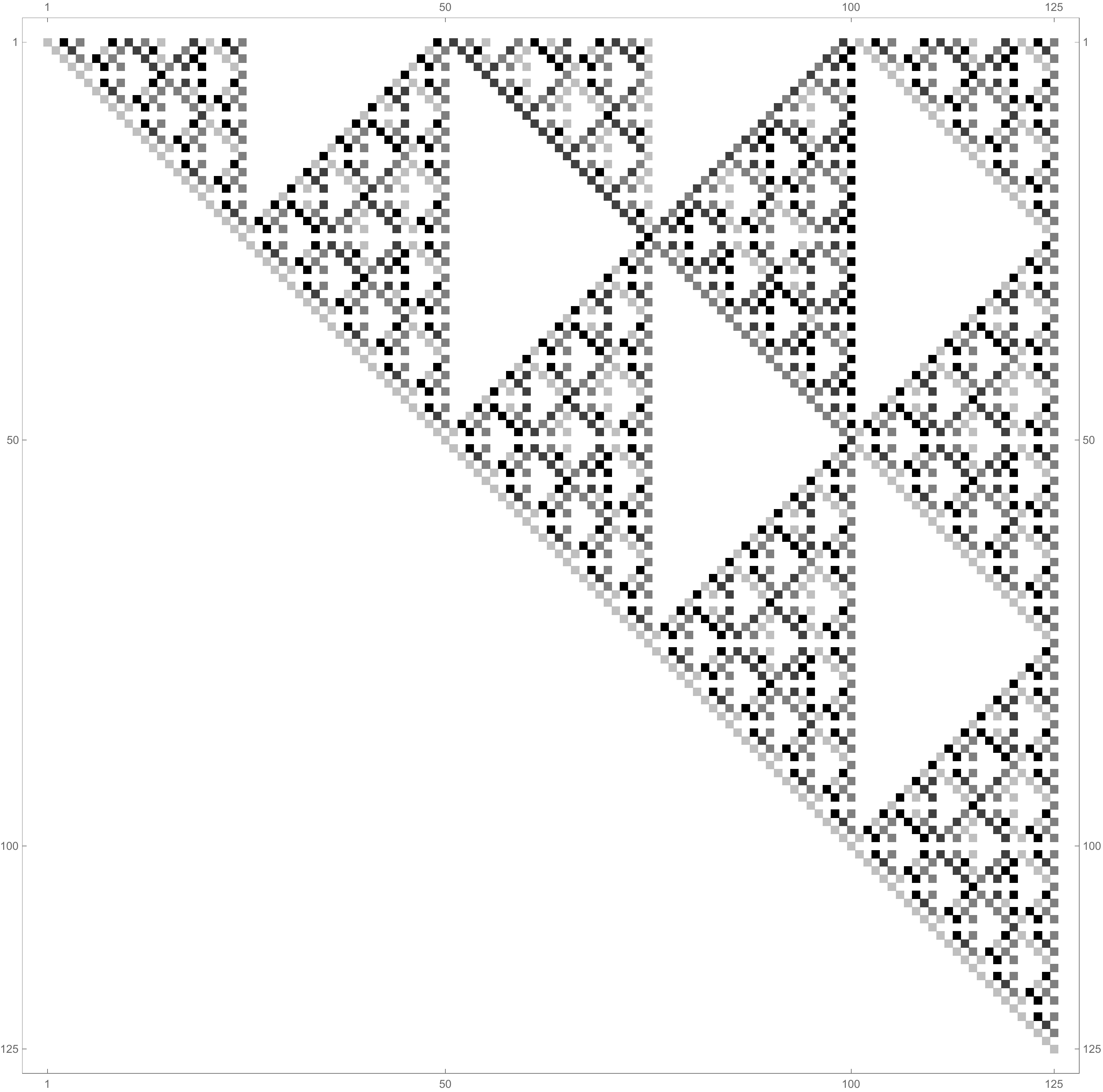}\\
			\includegraphics[width=0.18\textwidth]{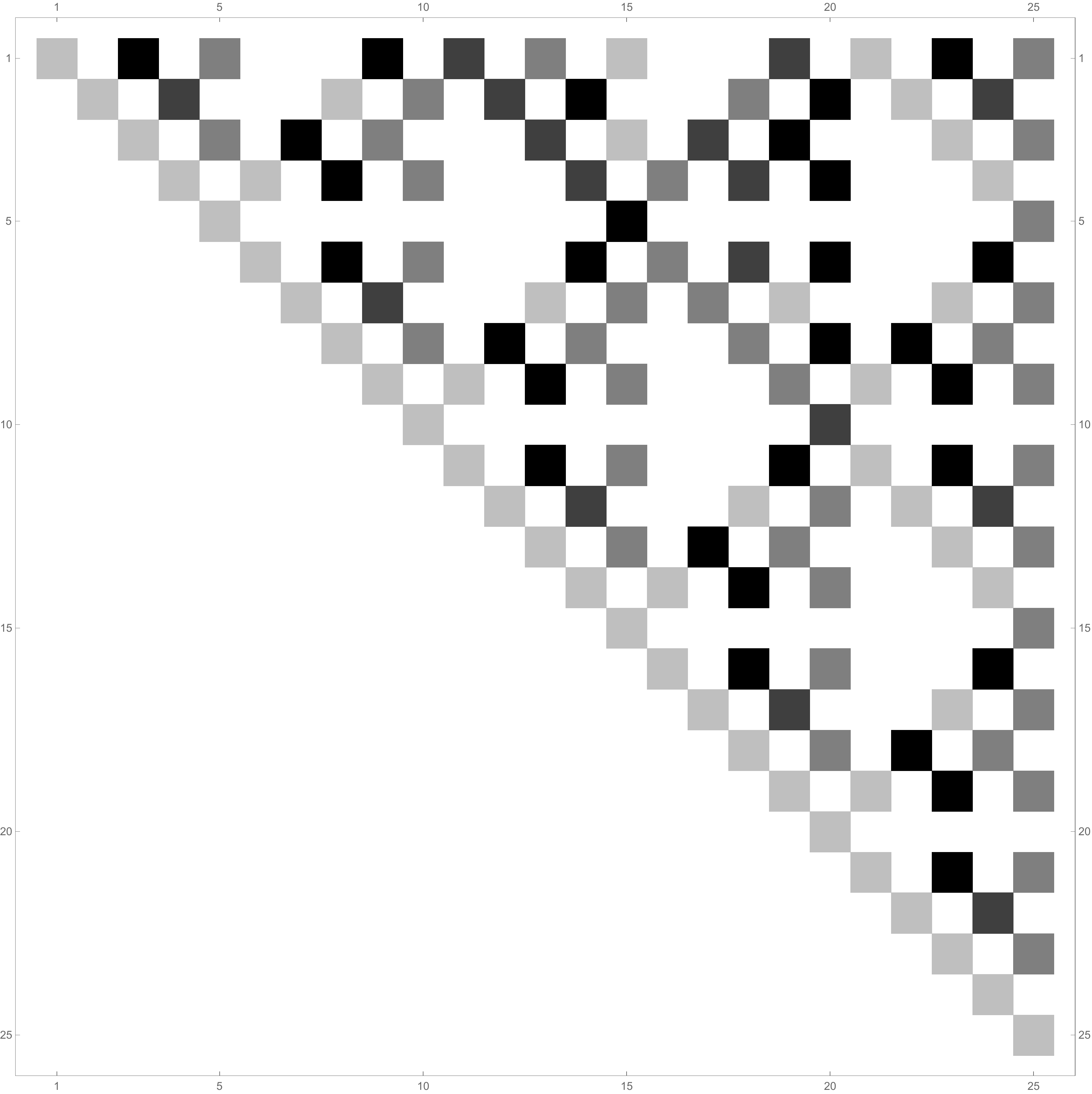}
		\includegraphics[width=0.18\textwidth]{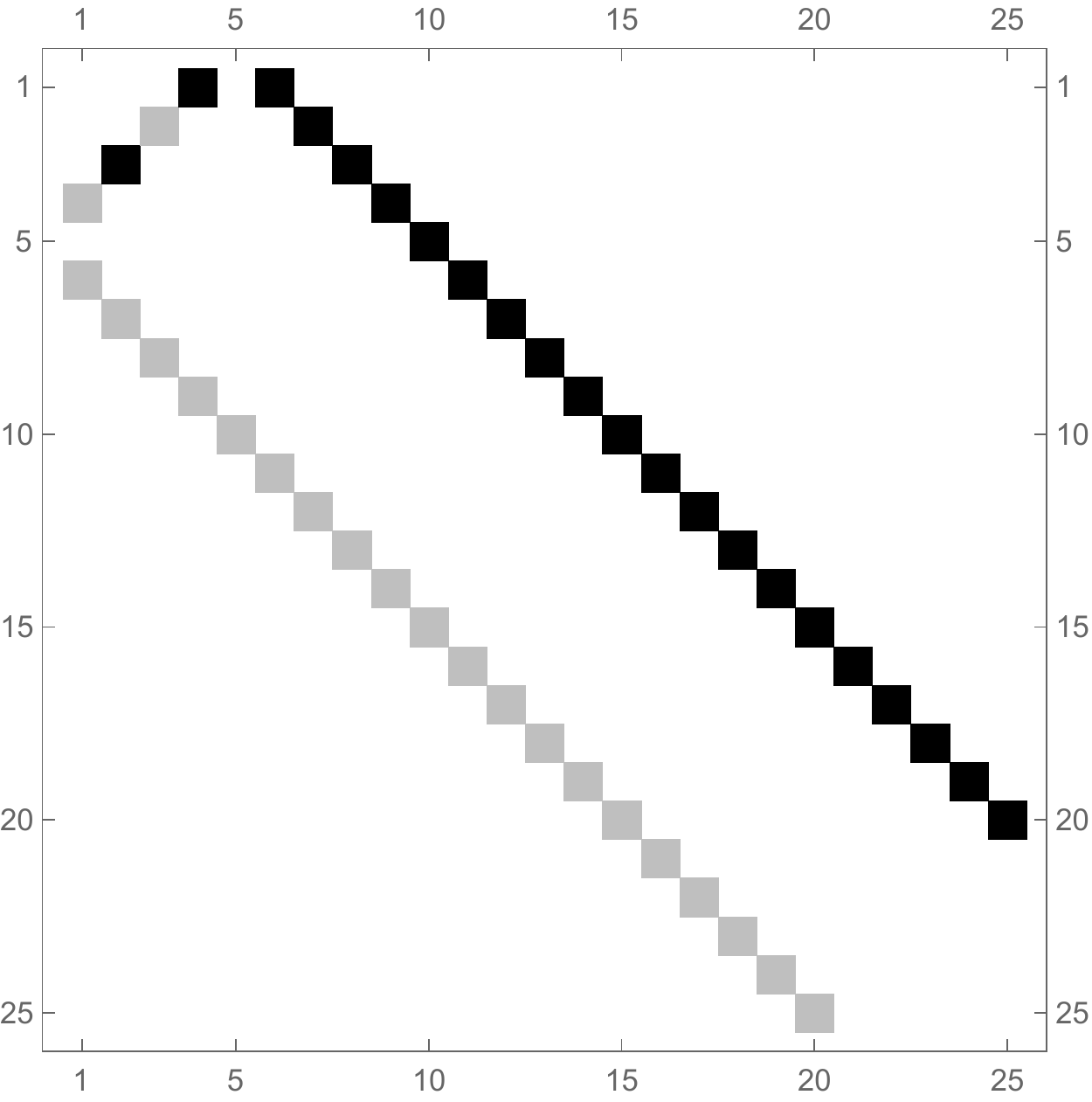} 
		\includegraphics[width=0.18\textwidth]{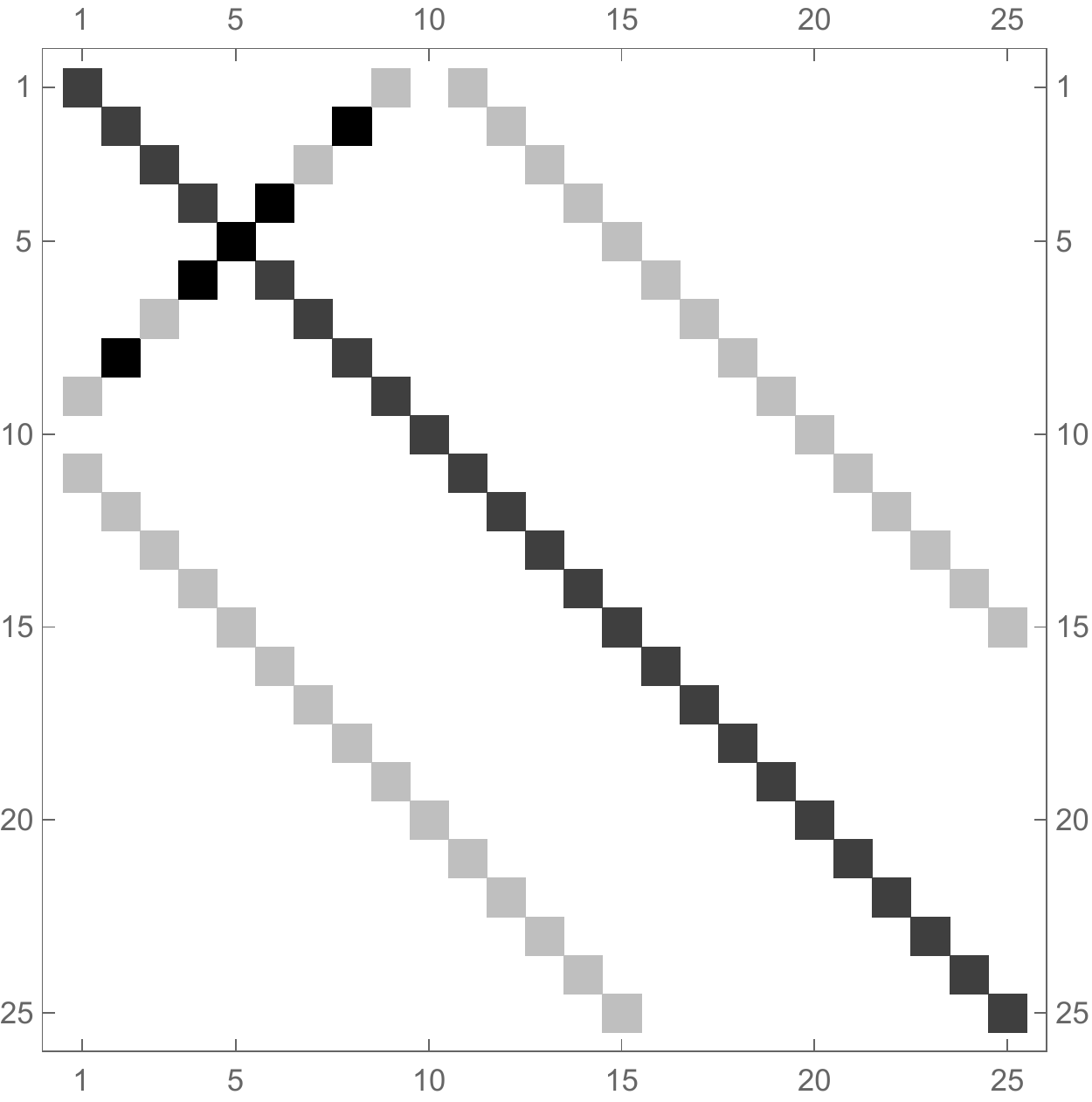} 
				\includegraphics[width=0.18\textwidth]{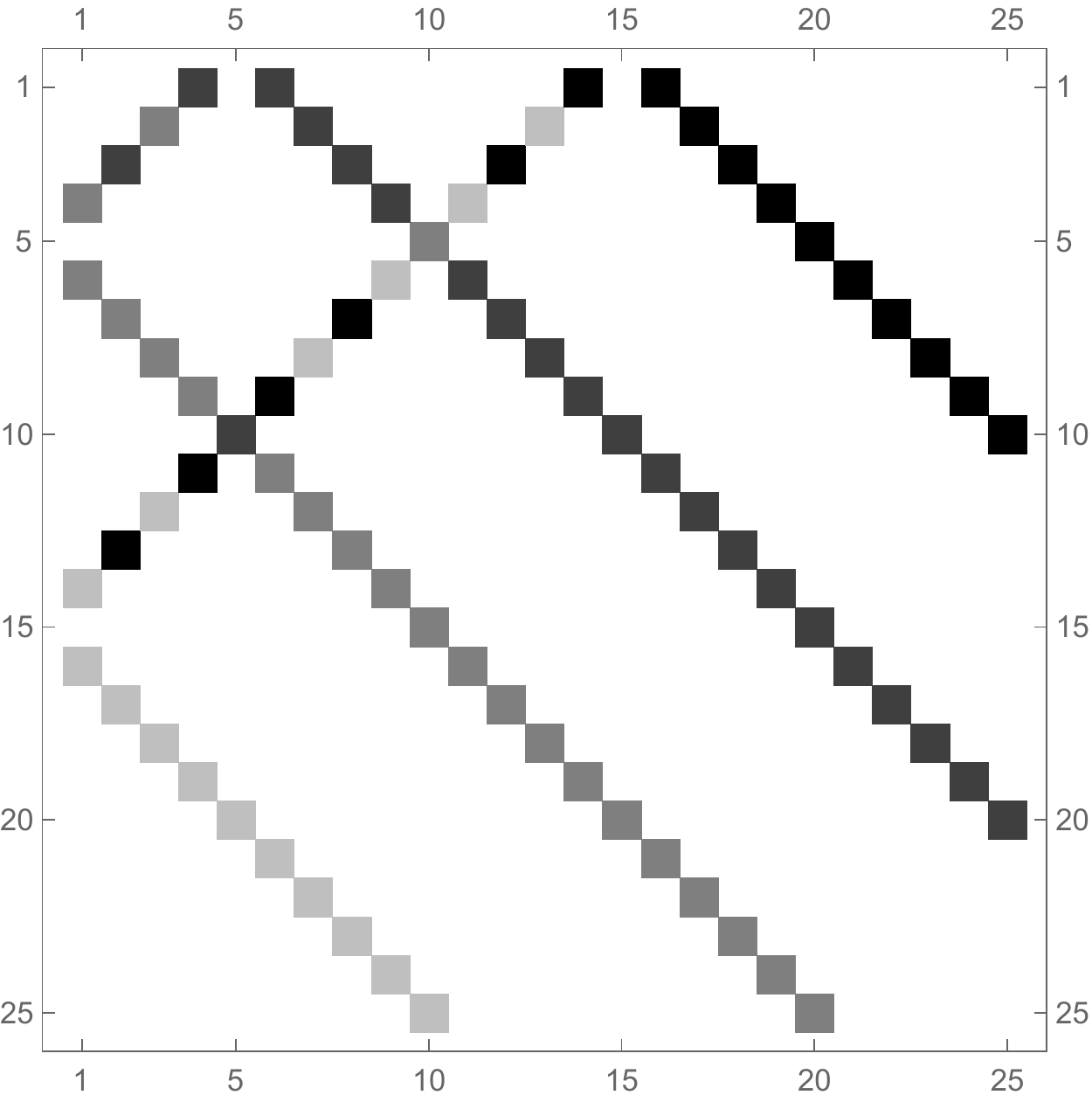}
						\includegraphics[width=0.18\textwidth]{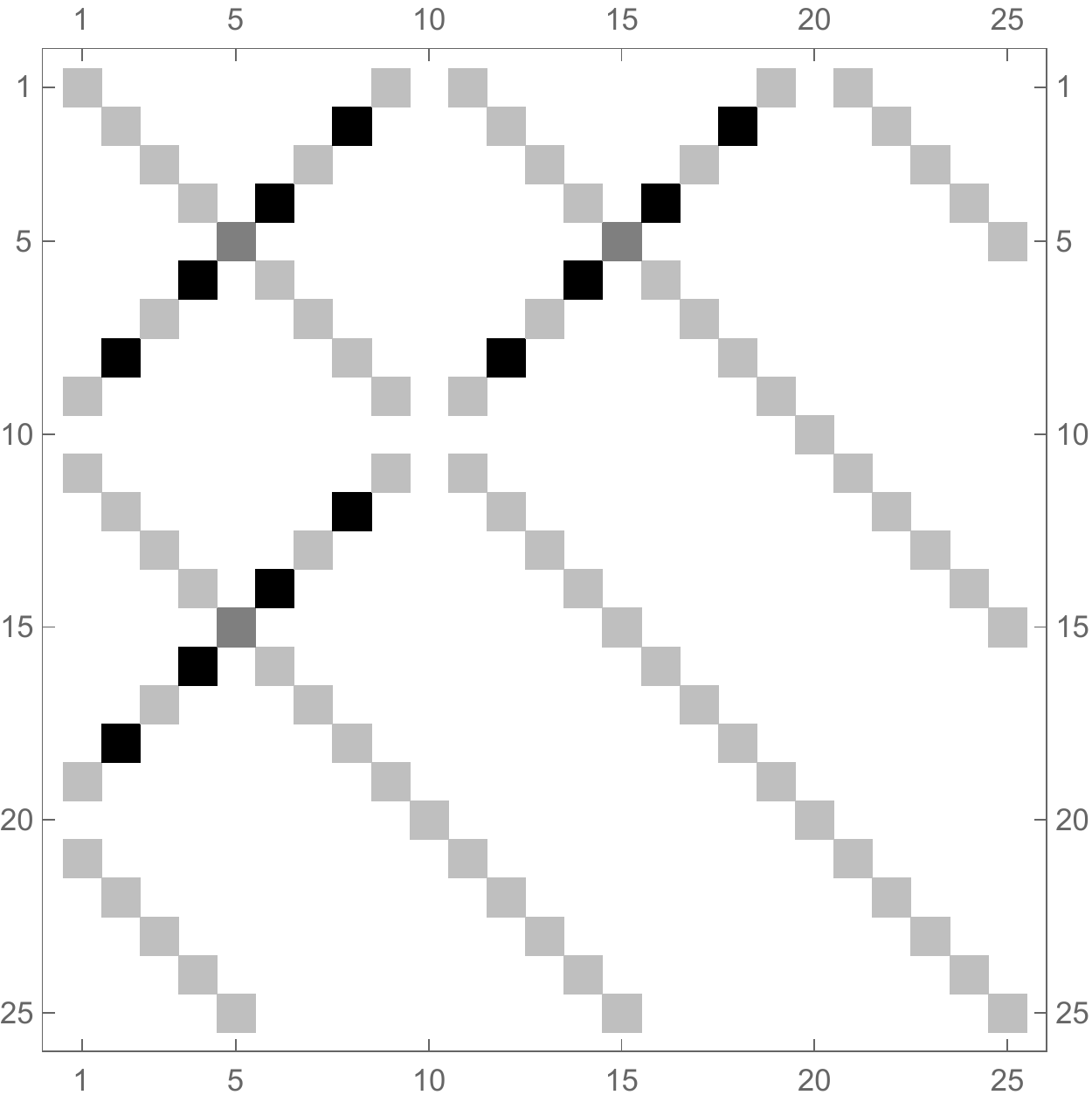} 
	\label{fig:p5}\caption{$[\mathcal{U}_{\phi_{1,0}}]_{5^3},\,[\mathcal{U}_{\phi_{1,0}}]_{5^2},\,[\mathcal{R}_{\phi_{1,0}}^{1\cdot 5}]_{5^2}$, $ [\mathcal{R}_{\phi_{1,0}}^{2\cdot 5}]_{5^2}$, $[\mathcal{R}_{\phi_{1,0}}^{3\cdot 5}]_{5^2}$, and $ [\mathcal{R}_{\phi_{1,0}}^{4\cdot 5}]_{5^2}$.}
\end{figure}
Hence the starting string of coefficients is $(1,0,4,0,2)$, that is followed up by $(0,0,0,4,0)$, $3\cdot (1,0,4,0,2)$, $3\cdot(0,0,0,4,0)$ and $1\cdot (1,0,4,0,2)$. etc. }
\end{example}

Note that from the similar recursive generation of $\mathcal{L}_\phi$ compared to the one of $ \mathcal{U}_\phi$ in Definition~\ref{def:mat} we may derive a slightly different self similar structure in $\mathcal{L}_\phi$. 

We mentioned already that the Hankel matrices $\mathcal{M}_\phi$ are the generating matrices of the Kronecker type sequences associated with $\phi$. We would like to remark here that the powers of the so-called Pascal matrices appear as generating matrices of the Faure sequences \cite{faure}. Pascal matrices possess nice self similar structures and their entries are defined using binomial coefficients, $P=(\binom{j-1}{i-1}\pmod{p}\cdot 1)_{i\geq 1,j\geq 1}$, while the entries of the matrix $\mathcal{U}_{\phi_{0,1}}$ are determined by Catalan's triangular numbers modulo $p$. 
Figure~\ref{fig:pascal} shows this Pascal matrix versus our matrix $\mathcal{U}_{\phi_{0,1}}$ for $p=2,\,p=3,\,p=4$. 

\begin{figure}[htbp]
	\centering
		\includegraphics[width=0.30\textwidth]{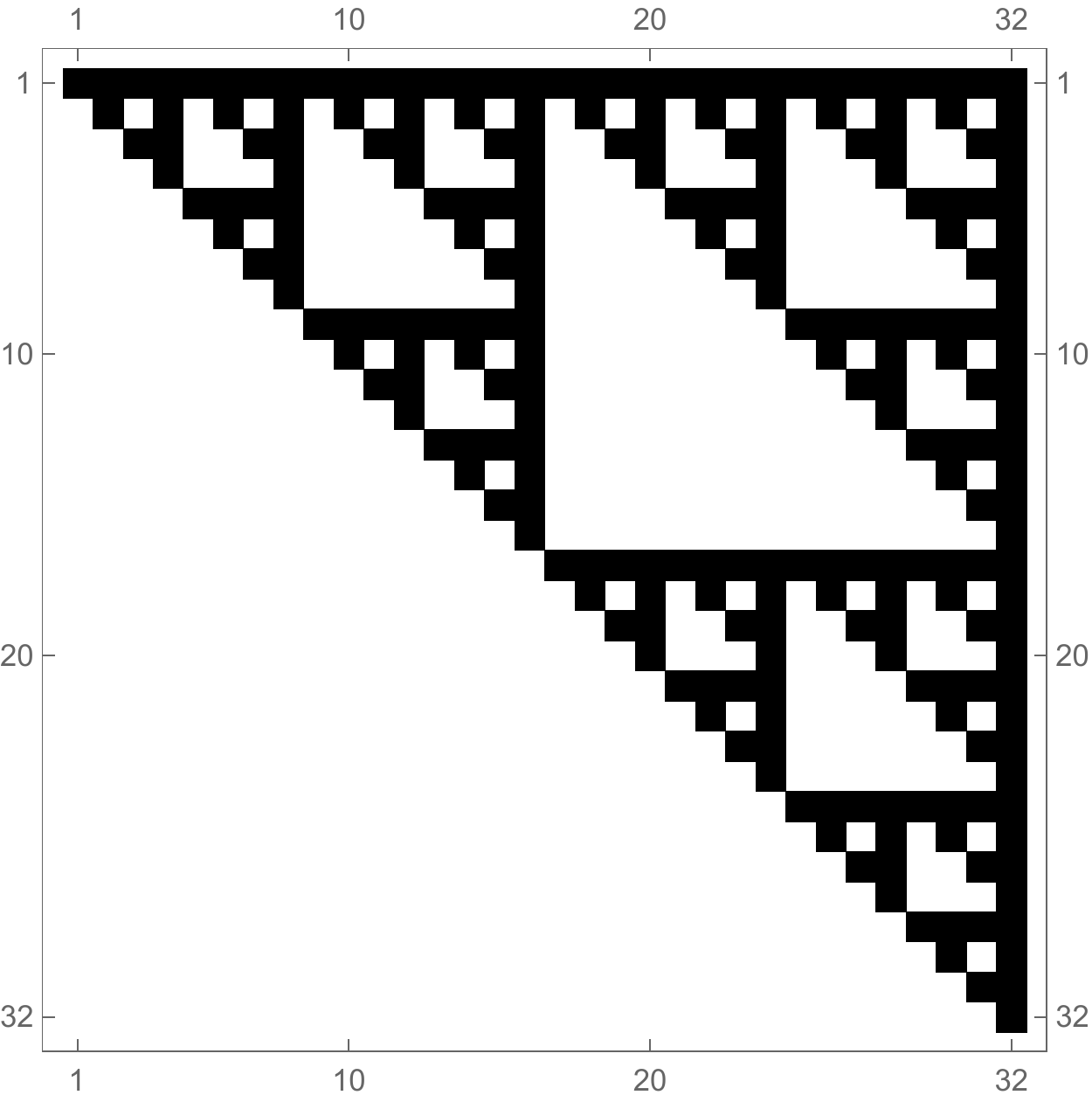}
		\includegraphics[width=0.30\textwidth]{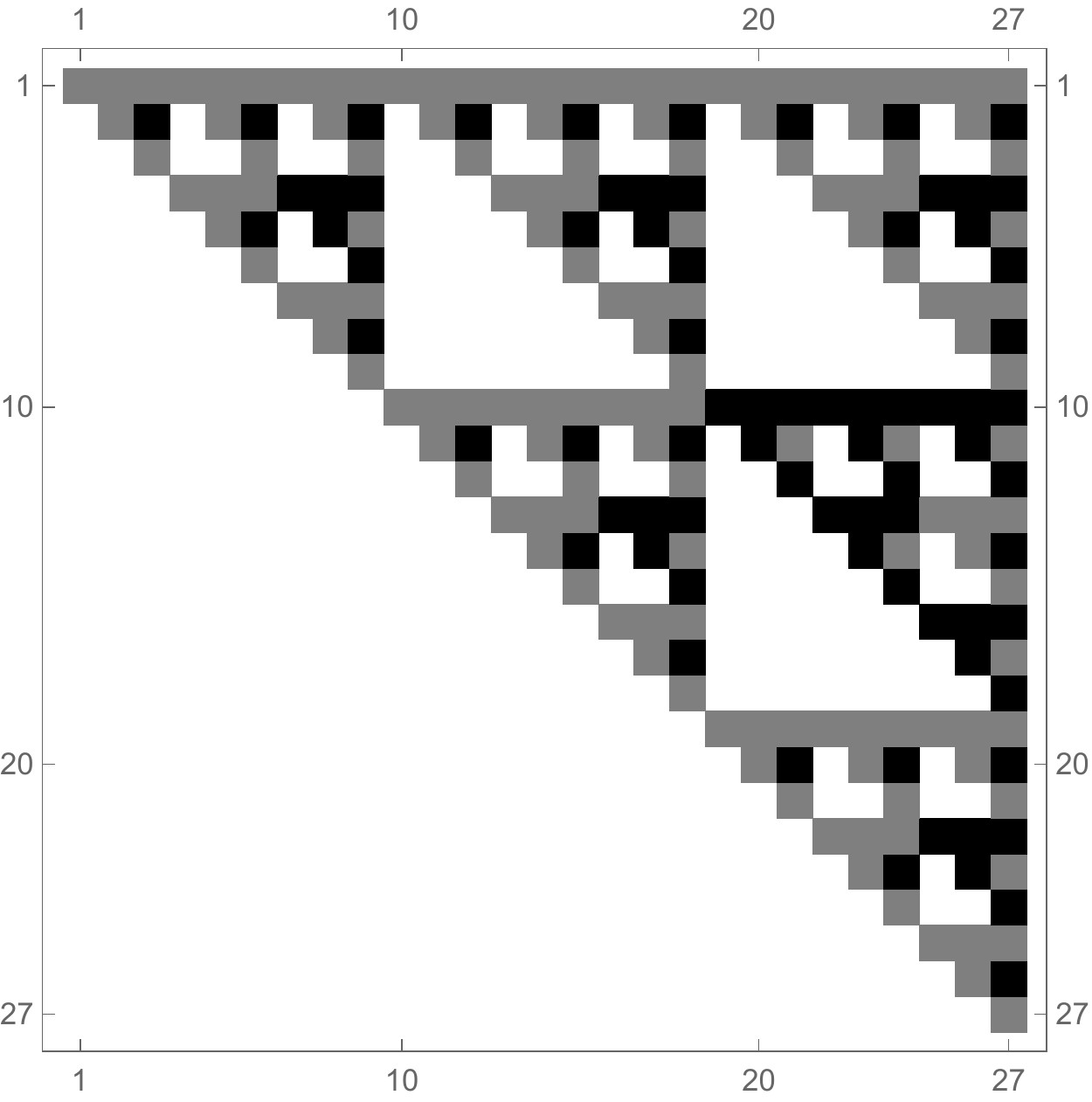}
		\includegraphics[width=0.30\textwidth]{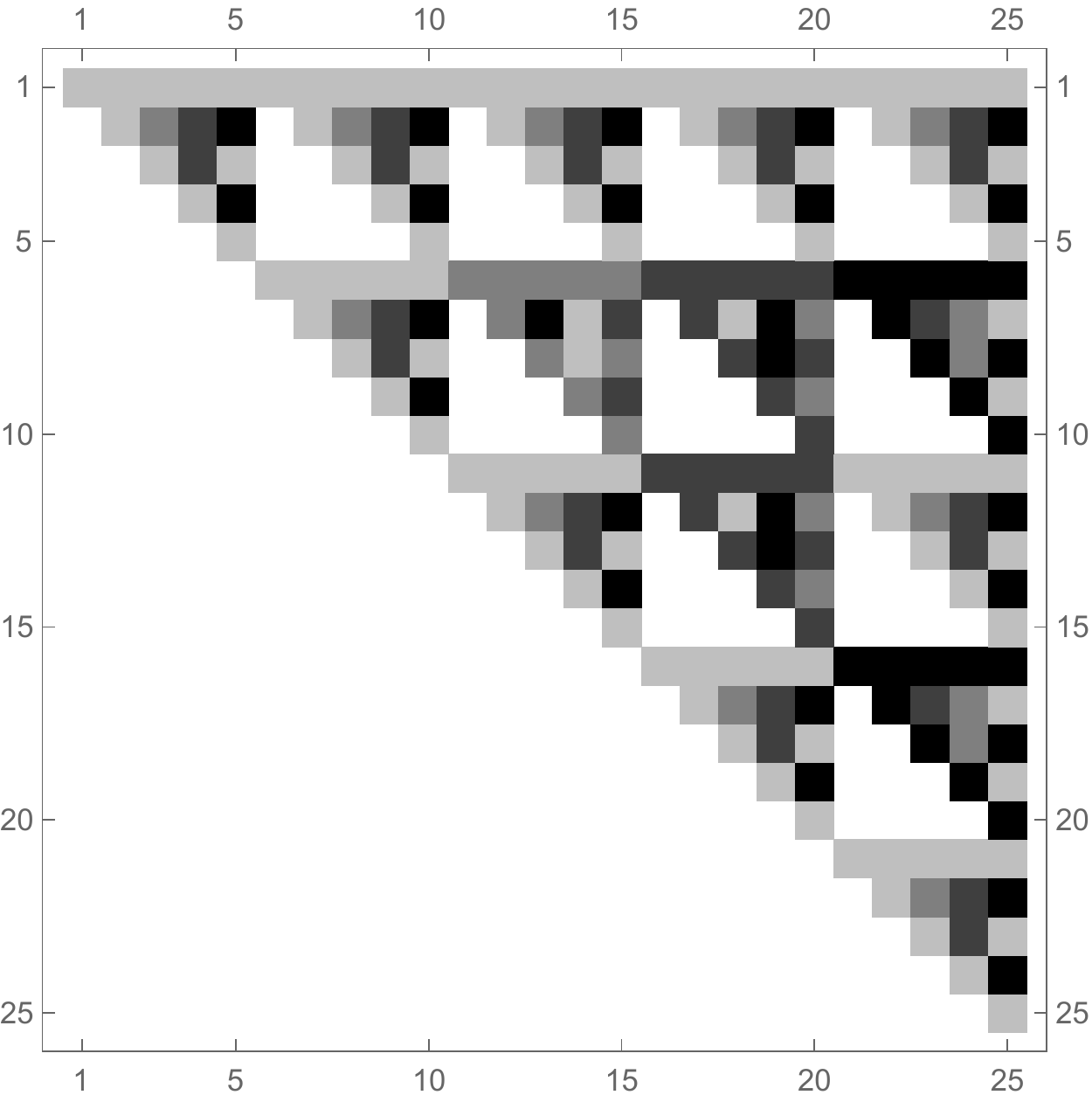}\\		
		\includegraphics[width=0.30\textwidth]{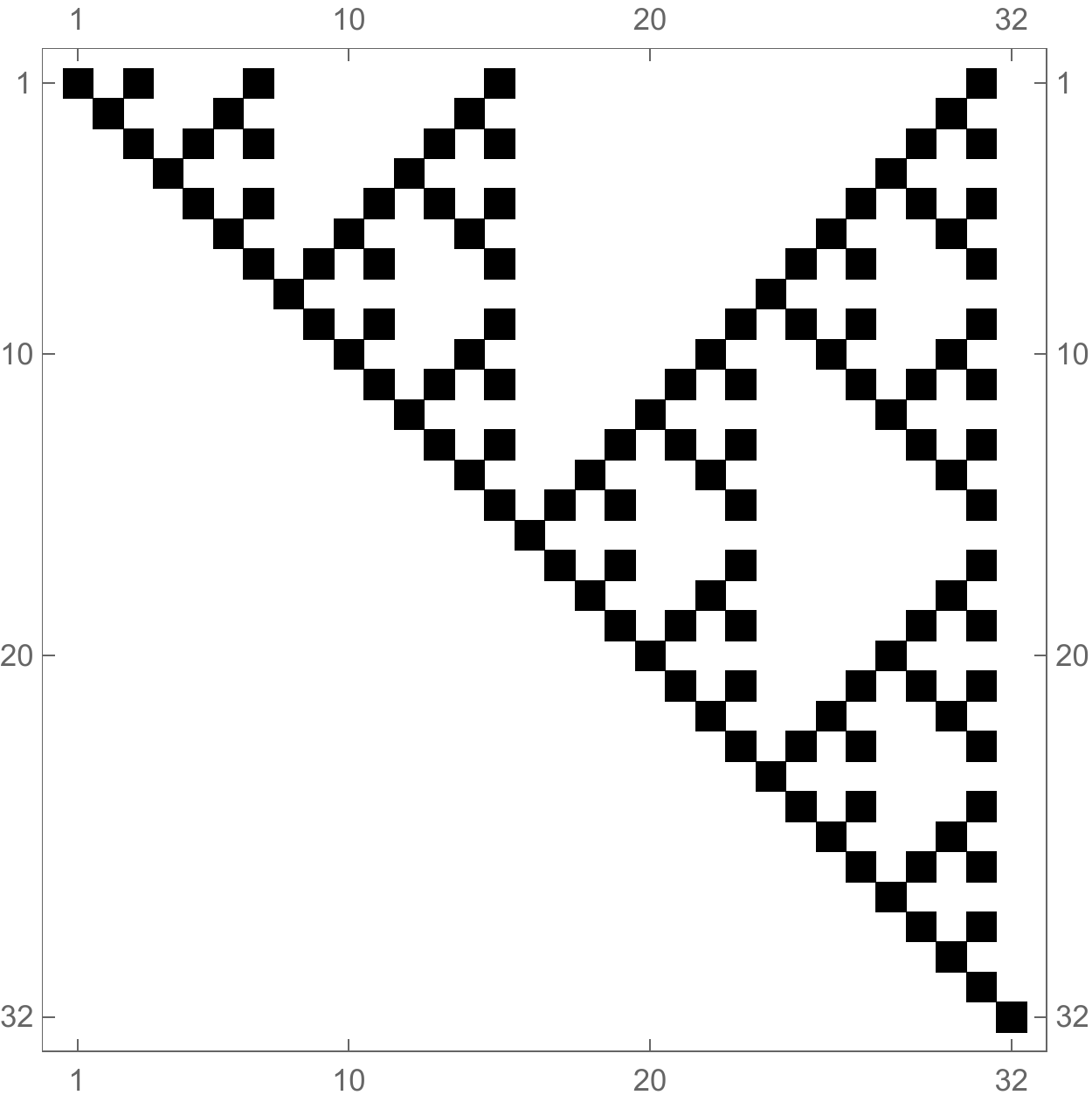}		
		\includegraphics[width=0.30\textwidth]{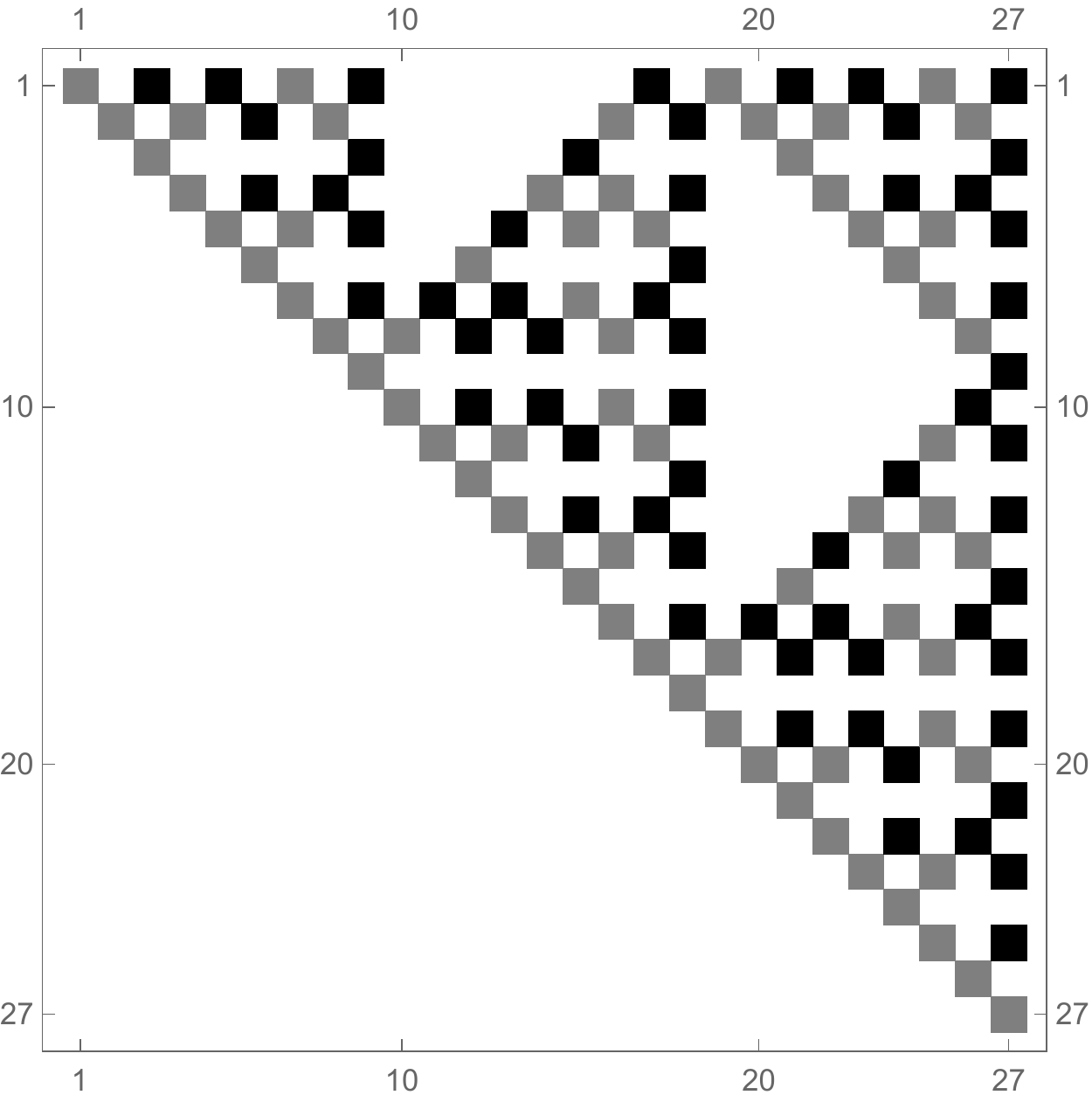}
		\includegraphics[width=0.30\textwidth]{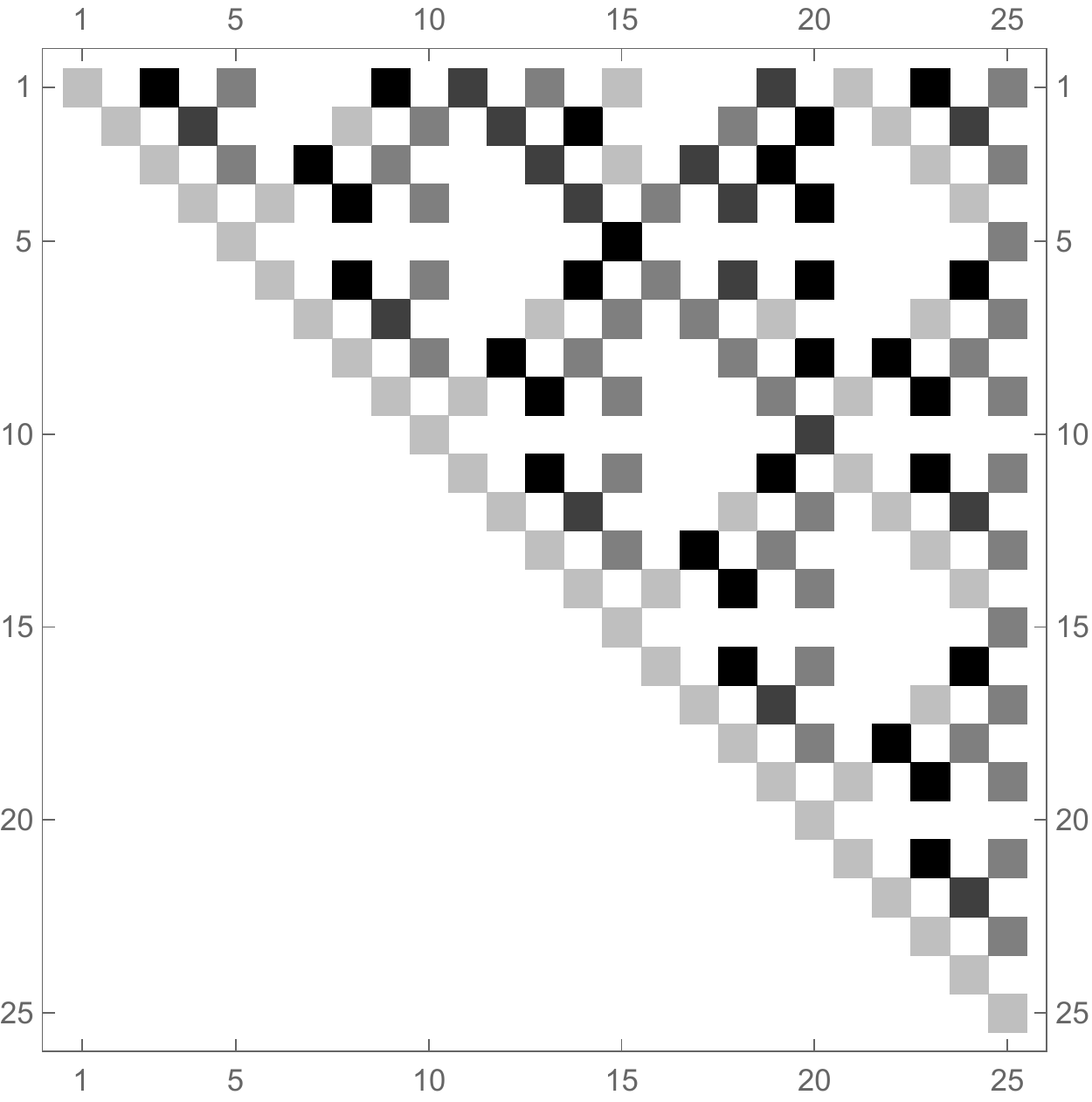}
	\caption{Upper left square submatrices of $P$ and $\mathcal{U}_{\phi_{1,0}}$ for $p=2$, $p=3$, and $p=5$.}
	\label{fig:pascal}
\end{figure}

\section{Discussions}

The content of this paper can be viewed as the starting point of various future investigations. 

One is the study of the $L_2$ discrepancy of the Kronecker type sequences generated by a Hankel matrix $\mathcal{M}_\phi$ associated with a golden ratio analog $\phi$. The main question is whether it is growing of optimal order $\sqrt{\log N}/N$ in $N$ or not. This study of the $L_2$ discrepancy was the starting point of the investigations in this paper. 

A further interesting problem is to generalize at least some of the results in this paper to more general $L\in\kk((X^{-1}))$. 

We mentioned already that the boundedness of the degrees of the continued fractions coefficients of $L$ is linked to regularities of the Hankel matrices, which control the distribution quality of the Kronecker type sequences associated to $L$. An interesting aspect is that the study of the multi-dimensional Kronecker type sequences is therefore linked to the multidimensional Diophantine approximation in the field of power series which is an active area of research (see exemplary \cite{BZ} and the references therein) . Moreover, sequences which are constructed via the digital method and have excellent distribution properties are also actively studied (see exemplary \cite{Hof} and the references therein). An exciting aspect is whether similar links can be discovered in the multidimensional situations that might enrich both research fields, the one of multi-dimensional Diophantine approximation in the field of power series and the one of distribution properties of digital sequences.




\section*{Acknowledgments}

The author is supported by the Austrian Science Fund (FWF), Project F5505-N26, which is a
part of the Special Research Program Quasi-Monte Carlo Methods: Theory and Applications.

\begin{small}
\noindent{\scshape{Authors' Address:}}\\

\noindent Roswitha Hofer:\\
Institute of Financial Mathematics and Applied Number Theory,\\
University of Linz\\
Altenbergerstr.~69\\
A-4040 Linz\\
Austria\\
E-mail: \texttt{roswitha.hofer@jku.at}
\end{small}

\end{document}